\documentclass[a4paper,11pt]{amsart}
\usepackage[mathscr]{eucal}
\usepackage{fancybox}
\usepackage{amscd}
\usepackage{amsmath}
\usepackage{amssymb}
\usepackage{amsthm}
\usepackage{float}
\usepackage{graphicx}
\usepackage{tikz}
\usepackage[all]{xy}
\usepackage{color}
\usepackage[linktocpage,hidelinks, hyperfootnotes=false]{hyperref}
\usepackage{array}
\usepackage{amscd}

\DeclareFontFamily{U}{rsfs}{%
\skewchar\font127}
\DeclareFontShape{U}{rsfs}{m}{n}{%
<-6>rsfs5<6-8.5>rsfs7<8.5->rsfs10}{}
\DeclareSymbolFont{rsfs}{U}{rsfs}{m}{n}
\DeclareSymbolFontAlphabet
{\mathrsfs}{rsfs}
\DeclareRobustCommand*\rsfs{%
\@fontswitch\relax\mathrsfs}

\theoremstyle{plain}
\newtheorem{thm}{Theorem}[section]
\newtheorem*{thm*}{Theorem}
\newtheorem{prop}[thm]{Proposition}

\newtheorem{lem}[thm]{Lemma}

\newtheorem{defi}[thm]{Definition}
\newtheorem{rmk}[thm]{Remark}
\newtheorem{cor}[thm]{Corollary}
\newtheorem*{cor*}{Corollary}

\newtheorem{prop-defi}[thm]{Proposition-Definition}
\newtheorem{thm-defi}[thm]{Theorem-Definition}
\newtheorem{defi-thm}[thm]{Definition-Theorem}
\newtheorem{lem-defi}[thm]{Lemma-Definition}

\newtheorem*{question*}{Question}

\newtheorem{conj}[thm]{Conjecture}

\newtheorem{setup-def}[thm]{Setup-Definition}

\newdimen\argwidth
\def\db[#1\db]{
 \setbox0=\hbox{$#1$}\argwidth=\wd0
 \setbox0=\hbox{$\left[\box0\right]$}
  \advance\argwidth by -\wd0
 \left[\kern.3\argwidth\box0 \kern.3\argwidth\right]}

\newcommand{\aA}{\mathcal{A}}
\newcommand{\bB}{\mathcal{B}}
\newcommand{\cC}{\mathcal{C}}

\newcommand{\eE}{\mathcal{E}}
\newcommand{\fF}{\mathcal{F}}

\newcommand{\hH}{\mathcal{H}}

\newcommand{\kK}{\mathcal{K}}

\newcommand{\mM}{\mathcal{M}}
\newcommand{\nN}{\mathcal{N}}
\newcommand{\oO}{\mathscr{O}}

\newcommand{\uU}{\mathcal{U}}

\newcommand{\Ob}{\mathcal{O}b}

\newcommand{\bL}{\mathbb{L}}
\newcommand{\bQ}{\mathbb{Q}}

\newcommand{\fG}{\mathfrak{g}}

\newcommand{\fc}{\mathfrak{c}}
\newcommand{\fn}{\mathfrak{n}}

\renewcommand{\tilde}{\widetilde}

\newcommand{\coker}{\mathop{\rm coker}\nolimits}

\newcommand{\lr}{\longrightarrow}

\newcommand{\id}{\textrm{id}}
\newcommand{\ch}{\mathop{\rm ch}\nolimits}

\newcommand{\td}{\mathop{\rm td}\nolimits}

\newcommand{\Spec}{\mathop{\rm Spec}\nolimits}

\newcommand{\Coh}{\mathop{\rm Coh}\nolimits}

\newcommand{\Sym}{\mathop{\rm Sym}\nolimits}

\newcommand{\GL}{\mathop{\rm GL}\nolimits}

\newcommand{\bC}{\mathbb{C}}

\newcommand{\bT}{\mathbb{T}}

\def\lal{_\lambda}
\def\vir{\mathrm{\vir}}
\def\loc{\mathrm{\loc}}

\def\lra{\longrightarrow}

\def\sO{\mathscr O}

\def\lalp{_\alpha}

\def\virt{^{\mathrm{vir}}}

\def\beq{\begin{equation}}
\def\eeq{\end{equation}}

\def\lalp{_\alpha }

\def\lab{_{\alpha\beta}}
\def\lbet{_\beta}
\def\mapright#1{\,\smash{\mathop{\lra}\limits^{#1}}\,}

\def\loc{{\mathrm{loc}}}

\makeatletter
 
  \@addtoreset{equation}{section}
\makeatother

\makeatletter
\def\@tocline#1#2#3#4#5#6#7{\relax
  \ifnum #1>\c@tocdepth 
  \else
    \par \addpenalty\@secpenalty\addvspace{#2}%
    \begingroup \hyphenpenalty\@M
    \@ifempty{#4}{%
      \@tempdima\csname r@tocindent\number#1\endcsname\relax
    }{%
      \@tempdima#4\relax
    }%
    \parindent\z@ \leftskip#3\relax \advance\leftskip\@tempdima\relax
    \rightskip\@pnumwidth plus4em \parfillskip-\@pnumwidth
    #5\leavevmode\hskip-\@tempdima
      \ifcase #1
       \or\or \hskip 1em \or \hskip 2em \else \hskip 3em \fi%
      #6\nobreak\relax
    \hfill\hbox to\@pnumwidth{\@tocpagenum{#7}}\par
    \nobreak
    \endgroup
  \fi}
\makeatother

\title[Riemann--Roch for APOT]{Virtual Riemann--Roch Theorems for Almost Perfect Obstruction Theories}

\author{Michail Savvas}
\address{Department of Mathematics, The University of Texas at Austin\\ 2515 Speedway, Austin, TX 78712, USA}
\email{msavvas@utexas.edu}


\begin{document}

\maketitle

\begin{abstract} 
This is the third in a series of works devoted to constructing virtual structure sheaves and $K$-theoretic invariants in moduli theory. The central objects of study are almost perfect obstruction theories, introduced by Y.-H. Kiem and the author as the appropriate notion in order to define invariants in $K$-theory for many moduli stacks of interest, including generalized $K$-theoretic Donaldson--Thomas invariants.

In this paper, we prove virtual Riemann--Roch theorems in the setting of almost perfect obstruction theory in both the non-equivariant and equivariant cases, including cosection localized versions. These generalize and remove technical assumptions from the virtual Riemann--Roch theorems of Fantechi--G\"{o}ttsche and Ravi--Sreedhar. The main technical ingredients are a treatment of the equivariant $K$-theory and equivariant Gysin map of sheaf stacks and a formula for the virtual Todd class.
\end{abstract}

\tableofcontents

\def\cG{{\mathcal{G}}}
\def\cF{{\mathcal{F}}}
\def\cK{{\mathcal{K}}}
\def\DM{Deligne--Mumford }
\def\sO{{\mathscr O}}
\def\Coh{\mathrm{Coh} }
\def\coh{\underline{\mathrm{Coh}} }
\def\QQ{\mathbb{Q} }
\def\PP{\mathbb{P} }
\def\AA{\mathbb{A} }

\section{Introduction}

The main approach in obtaining enumerative invariants in algebraic geometry consists of setting up a moduli stack parametrizing the objects of interest and then constructing a fundamental class on the stack, against which one can integrate cohomology classes that capture geometric constraints in order to obtain numbers. However, it is typical that, while these moduli stacks often have an explicit expected dimension, they are very singular, far from equidimensional and do not behave well under deforming the moduli problem.

To rectify these phenomena, Li--Tian \cite{LiTian} and Behrend--Fantechi \cite{BehFan} developed the theory of virtual fundamental cycles, which are a generalization of the usual fundamental class for a smooth scheme. They have been instrumental in defining and investigating several algebro-geometric enumerative invariants of great importance, such as Gromov--Witten \cite{BehGW}, Donaldson--Thomas \cite{Thomas} and Pandharipande--Thomas \cite{PT1} invariants, and are still one of the major components in modern enumerative geometry.

Their construction goes as follows: Any Deligne--Mumford stack $X$ is equipped with its intrinsic normal cone $\cC_X$ which locally for an \'{e}tale map $U \to X$ and a closed embedding $U \hookrightarrow M$ into a smooth scheme $M$ is the quotient stack $[C_{U/M} / T_M|_U]$ (cf. \cite{BehFan}). A perfect obstruction theory, defined in \cite{BehFan}, is a morphism $\phi \colon E \to \bL_X$ inducing an embedding of $\cC_X$ into the vector bundle stack $\eE_X = h^1 / h^0 ( E^\vee )$. Given this gadget, the virtual fundamental cycle is constructed as the intersection of the zero section $0_{\eE_X}$ with $\cC_X$
\begin{align*}
[X]\virt := 0_{\eE_X}^! [\cC_X] \in A_\ast(X).
\end{align*}

The existence of the virtual fundamental cycle is sufficient for constructing intersection theoretic and numerical enumerative invariants. Recently, there has been increased attention in such invariants that go beyond numbers or the intersection theory of cycles. Motivated by theoretical physics and geometric representation theory, invariants in $K$-theory are particularly desirable to have (see \cite{Okou2, Okou1}).

When the moduli stack $X$ admits a perfect obstruction theory with a global presentation $E = [E^{-1} \to E^0]$, where $E^{-1}, E^0$ are locally free sheaves on $X$, one can define the virtual structure sheaf, following \cite{yplee}, as
\begin{align*}
    [\sO_X\virt] := 0_{E_1}^! [\oO_{C_1}] \in K_0(X)
\end{align*}
where $E_1 = (E^{-1})^\vee$ and $C_1 = \cC_X \times_{\eE_X} E_1$.

However, there are many moduli stacks of interest which do not admit perfect obstruction theories or global presentations thereof, including moduli of simple complexes \cite{Inaba} and modified blowups of stacks of semistable sheaves and perfect complexes on Calabi--Yau threefolds \cite{KLS, Sav}.

In earlier work \cite{KiemSavvas, KiemSavvasLoc}, we developed a relaxed version of perfect obstruction theories, called almost perfect obstruction theories, which arise in the above moduli stacks, together with an extension of many of the usual tools, including virtual torus localization and cosection localization. An almost perfect obstruction theory admits an obstruction sheaf $\Ob_X$, which is the analogue of the sheaf $h^1(E^\vee)$, and induces an embedding of the coarse moduli sheaf $\fc_X$ of $\cC_X$ into $\Ob_X$ enabling us to define the virtual structure sheaf of $X$ as
\begin{align} \label{loc 1.1}
    [\sO_X\virt] := 0_{\Ob_X}^! [\oO_{\fc_X}] \in K_0(X).
\end{align}

To make sense of this expression, one of the main technical advances in \cite{KiemSavvas, KiemSavvasLoc} was the introduction of a $K$-theory of coherent sheaves on sheaf stacks and a Gysin map for sheaf stacks. 

At the same time, the embedding $\fc_X \to \Ob_X$ induces a cycle class $[\fc_X] \in A_\ast(\Ob_X)$, defined by Chang-Li in \cite{LiChang}, and hence, using the intersection theoretic Gysin map of Chang-Li for the sheaf stack $\Ob_X$, a virtual fundamental cycle
\begin{align} \label{loc 1.2}
    [X]\virt := 0_{\Ob_X}^! [\fc_X] \in A_*(X).
\end{align}

\medskip

It is a natural question whether the virtual structure sheaf~\eqref{loc 1.1} and the virtual fundamental cycle~\eqref{loc 1.2} induced by an almost perfect obstruction theory are related. Since these generalize the structure sheaf and fundamental cycle of a smooth scheme, it is natural to expect that such a relation uses the Riemann--Roch transformation, whenever this is defined.

In the presence of a perfect obstruction theory, virtual Riemann--Roch theorems providing the desired relation have been established for algebraic spaces by Fantechi--G\"{o}ttsche \cite{VirRR} and for quotient stacks by Ravi--Sreedhar \cite{EquivVirRR}. Algebraic spaces and, more generally, quotient stacks are two important classes of algebraic stacks with a well-defined Riemann--Roch transformation. 

In this paper, we prove a more general virtual Riemann--Roch theorem in the setting of almost perfect obstruction theory. Our main result can be summarized as follows.

\begin{thm*}
Let $G$ be an algebraic group, $X$ an algebraic space with a $G$-action, $Y$ a smooth scheme of pure dimension with a $G$-action and $f \colon X \to Y$ a $G$-equivariant morphism equipped with a $G$-equivariant almost perfect obstruction theory $\phi$ (cf. Definition~\ref{equivariant APOT}).

The virtual structure sheaf $[\sO_X\virt]$ and virtual fundamental cycle $[X]\virt$  are then naturally $G$-equivariant (cf. Definition~\ref{equivariant vir str sheaf and cycle def}) and satify the equivariant Riemann--Roch formula
\begin{align*}
    \tau_X^G ([\sO_X\virt]) = \td^G(T\virt_{X/Y}) \cap [X]\virt \in A_*^G(X).
\end{align*}

Here $\tau_X^G$ is the equivariant Riemann--Roch transformation of Edidin--Graham \cite{EdidinGrahamRR} and $\td^G(T\virt_{X/Y})$ is defined by the formula (cf. Theorem~\ref{vir todd class formula thm})
\begin{align*}
    \td^G(T_{X/Y}\virt) = \frac{\td^G(E_0)}{\td^G(E_1)} \cdot \frac{\td^G(f^\ast T_Y)}{\td^G(\fG)},
\end{align*}
whenever $X$ admits a $G$-invariant closed embedding into a smooth algebraic space with a $G$-action and $[E_0 \to E_1]$ is a $G$-equivariant two-term complex of locally free sheaves on $X$ which gives a global presentation of the virtual tangent bundle of $f$ (cf. Definition~\ref{equiv global prep of vir tangent}). $\fG$ denotes the Lie algebra of $G$, considered as a trivial  $G$-equivariant vector bundle on $X$ under the adjoint action of $G$.
\end{thm*}

As a corollary, we obtain the following virtual Grothendieck--Riemann--Roch formula.

\begin{cor*}
When $f$ is proper, we have for any $V \in K_G^0(X)$
\begin{align*}
    \mathrm{ch}^G(f_\ast(V\otimes [\sO_X\virt])) \cdot \td^G(T_Y) \cap [Y] = f_\ast \left( \mathrm{ch}^G(V) \cdot \td^G(T\virt_{X/Y}) \cap [X]\virt \right).
\end{align*}
\end{cor*}

When the almost perfect obstruction theory $\phi$ admits a $G$-equivariant cosection $\Ob_X \to \oO_X$, we also obtain cosection localized versions of the above statements.
\medskip

Our virtual Riemann--Roch theorem is a generalization of \cite{VirRR, EquivVirRR} in multiple ways, as it only uses the existence of an almost perfect obstruction theory, and even in the case of a perfect obstruction theory requires a weaker notion of a global resolution of the two-term complex $E$ in its definition.

In fact, using the results of \cite{HekSav}, our theorem applies to the moduli stacks that are used to define intersection-theoretic and $K$-theoretic generalized Donaldson--Thomas invariants of sheaves on Calabi--Yau threefolds in \cite{KLS, KiemSavvas}, giving an explicit relationship between the two. This is a setting which is not covered by the Riemann--Roch theorems established in \cite{VirRR, EquivVirRR}.

Besides these, an essential technical input to the theorem of independent interest is the development of an equivariant version of the $K$-theory of coherent sheaves on sheaf stacks as well as an equivariant Gysin map. This is necessary in order to define a natural equivariant structure on the virtual structure sheaf $[\sO_X\virt]$.

When $G$ is the trivial group, the theorem specializes to a non-equivariant virtual Riemann--Roch theorem for an almost perfect obstruction theory on the morphism $f \colon X \to Y$.

A noteworthy feature of the theorem is the explicit formula it provides for $\td^G(T_{X/Y}\virt)$ even though the virtual tangent bundle $T_{X/Y}\virt$ of $f$ does not a priori make sense for an almost perfect obstruction theory. It would be interesting to investigate whether such a formula exists under more general conditions. This essentially boils down to representing the class $[T_{X/Y}] - [\Ob_X] \in K_0(X)$ by a class in $K^0(X)$. We formulate the most optimistic such scenario in Conjecture~\ref{conj 5.4}.

Finally, we mention that variants of the virtual Riemann--Roch theorem for \DM stacks with perfect obstruction theory have been studied in \cite{Joshua1, KhanRR, ToenDMRR}, often utilizing stronger input from derived algebraic geometry. We expect our results to be compatible with these and apply to any \DM stack with a Riemann--Roch transformation $\tau_X \colon K_0(X) \to A_*(X)$ satisfying the usual properties.

\subsection*{Layout of the paper} \S\ref{background section} collects necessary background on the resolution property for stacks and Riemann--Roch theorems. In \S\ref{equivariant K-theory and Gysin maps background section} we study the equivariant $K$-theory of coherent sheaves on sheaf stacks and their equivariant Gysin maps. \S\ref{APOT section} defines equivariant almost perfect obstruction theories and their induced virtual structure sheaves and fundamental cycles. In \S\ref{RR APOT Section} we combine these ingredients to give a formula for the virtual Todd class and establish virtual Riemann--Roch theorems for almost perfect obstruction theories in the non-equivariant and equivariant cases. Finally, \S\ref{cos loc section} generalizes these results in the presence of a cosection of the obstruction sheaf.

\subsection*{Acknowledgements} 
We are grateful to Young-Hoon Kiem for our prior collaboration which gave birth to the machinery of almost perfect obstruction theory and many helpful discussions on the results of this paper and throughout the years. We would also like to sincerely thank Andrea Ricolfi for his interest and reviewing an earlier draft of this work.

\subsection*{Notation and conventions} 
Everything in this paper is over the field $\bC$ of complex numbers. All stacks are of finite type and Deligne--Mumford stacks are assumed to be separated. 

$G$ typically denotes a linear algebraic group and $\fG$ its Lie algebra.

Unless otherwise stated, $K_0(X)$ denotes the Grothendieck group of coherent sheaves on a stack $X$ with $\bQ$-coefficients and $K^0(X)$ denotes the Grothendieck group of locally free sheaves on $X$, also tensored with $\bQ$. Similarly, $A_*(X)$ denotes the direct product of all Chow groups of a stack $X$ with $\bQ$-coefficients. When $X$ admits an action by a group $G$, $K_0^G(X), K_G^0(X)$ and $A_*^G(X)$ denote the corresponding $G$-equivariant groups.

If $E$ is a locally free sheaf on a Deligne--Mumford stack $X$, we will use the term “vector bundle” to refer to its total space. If $\fF$ is a coherent sheaf on a Deligne--Mumford stack $X$, we will use the same letter to refer to the associated sheaf stack.

For a morphism $f:T\to X$ of stacks and a coherent sheaf $\fF$ on $X$, its pullback $f^*\fF$ is sometimes denoted by $\fF|_T$ when the map $f$ is clear from context.
The bounded derived category of coherent sheaves on a stack $X$ is denoted by $D(X)$ and $\bL_{X/Y}=L^{\ge -1}_{X/Y}$ denotes the truncated cotangent complex for a morphism $X \to Y$.

\section{Background on Equivariant Geometry and Riemann--Roch } \label{background section}

Throughout this section, $G$ denotes a linear algebraic group. $X$ will be a Deligne--Mumford stack with an action of $G$, $Y$ is a smooth Artin stack of pure dimension with a $G$-action, and $f \colon X \to Y$ a $G$-equivariant morphism. For a detailed account of the foundations of group actions on stacks, we refer the reader to \cite{Romagny}.

\subsection{The resolution property} 

The following definition explains several variants of the resolution property for stacks and introduces some useful terminology.

\begin{defi} \label{resolution property def}
We say that an Artin stack $\mM$ satisfies the resolution property if every coherent sheaf $\fF$ on $\mM$ is the quotient of a locally free sheaf.

We say that $\mM$ satisfies the resolution property \'{e}tale locally if there exists an \'{e}tale cover $\rho \colon \uU \to \mM$, called an atlas for $\mM$, such that $\uU$ satisfies the resolution property.

Finally, for a given coherent sheaf $\fF$, we say that $\mM$ satisfies the resolution property \'{e}tale locally for $\fF$ if there exists an \'{e}tale cover $\rho_\fF \colon \uU \to \mM$ such that the pullback $\rho_{\fF}^\ast \fF$ is the quotient of a locally free sheaf on $\uU$. In this case, the cover $\uU \to \mM$ is called an $\fF$-atlas for $\mM$.
\end{defi}

When $\mM = [X/G]$ is a quotient stack, we can slightly specialize the notion of an atlas as follows.

\begin{defi} \label{equivariant atlas def}
We say that $X$ admits a $G$-equivariant atlas if there exists a $G$-equivariant \'{e}tale cover $U \to X$ such that the stack $[U/G]$ satisfies the resolution property. We then say that $[U/G] \to [X/G]$ is an effective atlas for the quotient stack $\mM = [X/G]$.

Given a $G$-equivariant coherent sheaf $\fF$ on $X$, we say that $X$ admits a $G$-equivariant atlas for $\fF$ if there exists a $G$-equivariant \'{e}tale cover $\rho_\fF \colon U \to X$ such that $\rho_\fF^\ast \fF$ is the quotient of a $G$-equivariant locally free sheaf on $U$. We say in that case that $[U/G] \to [X/G]$ is an effective $\fF$-atlas for $\mM$.
\end{defi}

The following theorem is a collection of results in the literature which give conditions for the resolution property to hold on the nose or \'{e}tale locally.

\begin{thm} \label{theorem 2.3}\text{ }
\begin{enumerate}
    \item \cite{ThomasonRes, Totaro} If $X$ is a scheme that admits an ample family of $G$-equivariant line bundles, i.e., if $X$ is the union of open affine subsets of the form $\lbrace s \neq 0 \rbrace$ where $s$ is a section of a $G$-equivariant line bundle on $X$, then $[X/G]$ satisfies the resolution property. In particular, this is true when $X$ is quasi-affine.
    \item \cite{Gross, Totaro} A quasi-compact quasi-separated Artin stack $\mM$ is of the form $[X / \GL_N]$, where $X$ is a quasi-affine scheme, if and only if $\mM$ has affine stabilizers and satisfies the resolution property.
    \item \cite[Theorem~4.4]{Alper} If $X$ is an algebraic space with a $G$-action and every point of $X$ with closed $G$-orbit has reductive stabilizer, then $X$ admits an affine (as a morphism) $G$-equivariant atlas $U \to X$ and hence $[U/G] \to [X/G]$ gives an affine effective atlas for $[X/G]$. In particular, this is true when $[X/G]$ is \DM.
    \item \cite[Theorem~4.21]{Alper} If $\mM$ is an Artin stack whose closed points have reductive stabilizers, then $\mM$ satisfies the resolution property \'{e}tale locally. In particular, any \DM stack satisfies the resolution property \'{e}tale locally.
\end{enumerate}
\end{thm}

\subsection{Riemann--Roch for algebraic spaces} For every scheme $X$, Fulton \cite[Theorem~8.3]{Fulton} proves that there exists a Riemann--Roch transformation 
\begin{align} \label{rr transform}
    \tau_X \colon K_0(X) \lr A_*(X),
\end{align}
which is a group homomorphism satisfying the following properties:
\begin{enumerate}
    \item module homomorphism: for any $V \in K^0(X)$ and any $F \in K_0(X)$ one has $\tau_X ( V \otimes F) = \mathrm{ch}(V ) \cap \tau_X(F)$;
    \item Todd: if $X$ is smooth, $\tau_X(\sO_X) = \td(T_X) \cap [X]$; hence for every $V \in K^0(X)$ one has $\tau_X (V \otimes \sO_X) = \mathrm{ch}(V ) \cdot \td(T_X) \cap [X]$;
    \item covariance: for every proper morphism $f : X \to Y$ one has $$f_\ast \circ \tau_X = \tau_{Y} \circ f_\ast \colon K_0(X) \to A_\ast(Y);$$
    \item local complete intersection: if $f : X \to Y$ is an lci morphism, and $\alpha \in K_0(Y)$, then $f^\ast (\tau_Y (\alpha)) = (\td T_f )^{-1} \cap \tau_X (f^\ast \alpha)$.
\end{enumerate}

Motivated by (2), we define the Todd class of $X$ by $\td (X) := \tau_X (\sO_X)$.

By \cite{Gillet}, Fulton's arguments extend to establish the existence of $\tau_X$ for any algebraic space $X$.

\subsection{Riemann--Roch for quotient stacks} \label{RR for quotient sec} Suppose now that $X$ is an algebraic space with an action of a linear algebraic group $G$ with $\dim G = g$. Edidin-Graham \cite{EdidinGrahamRR} then construct an equivariant analogue of the Riemann--Roch transformation~\eqref{rr transform}
\begin{align} \label{equivariant rr transform}
    \tau_X^G \colon K_0^G(X) \lr A_\ast^G (X)
\end{align}
from the $K$-theory group of $G$-equivariant coherent sheaves on $X$ to the $G$-equivariant Chow groups of $X$, defined in \cite{EdidinGrahamEquivIT}. This can be equivalently viewed as a Riemann--Roch transformation for the quotient stack $[X/G]$, since there are natural isomorphisms
$$K_0^G(X) \cong K_0([X/G]),\ A_\ast^G (X) \cong A_\ast([X/G]), $$
where $A_\ast([X/G])$ denotes the Chow groups defined by Kresch in \cite{KreschCycles}.

We review the construction of $\tau_X^G$ since it will be necessary in order to prove the main result of this paper later on.
\medskip

Firstly, we recall the definition of $A_i^G(X)$. 

Let $V$ be a complex $l$-dimensional representation of $G$ and let $U$ be a $G$-invariant open subset of $V$ such that $G$ acts freely on $U$. The diagonal action on $X \times U$ is also free, so there is a quotient in the category of algebraic spaces $X \times U \to X \times^G U$, which in general may not be
a scheme. The pair $(V, U)$ is called an $l$-dimensional good pair for an integer $j$ if $\mathrm{codim}(U, V ) > j$.

The $i$-th equivariant Chow group of $X$ is defined as 
$$A_i^G(X) := A_{i+l-g}(X \times^G U),$$ 
where $(V, U)$ is an $l$-dimensional good pair for the integer $n - i$ and $A_\ast$ denotes the usual Chow group. As usual, we write $A_\ast^G(X) = \prod_i A_i^G(X)$.
\medskip

To define $\tau_X^G$, we specify its $i$-th component $(\tau_X^G)_i$ mapping to $A_i^G(X)$. As above, let $(V, U)$ be an $l$-dimensional good pair for the integer $n - i$. Write $\jmath \colon U \to V$ for the open inclusion and denote by the same letter the open inclusion $X \times U \to X \times V$. Write $\pi_V \colon X \times V \to X$ for the projection. We also have a natural projection map $X \times^G (U \times V) \to X \times^G U$, which defines a vector bundle over $X \times^G U$. Then $(\tau_X^G)_i$ is defined by the commutative diagram
\begin{align} \label{fund diagram equivariant rr}
    \xymatrix{
    K_0^G(X \times V) \ar[r]^-{\jmath^\ast} & K_0^G(X \times U) \ar[r] & K_0(X \times^G U) \ar[d]^-{\left( \frac{\tau_{X \times^G U}}{\td(X \times^G (U \times V))} \right)_{i+l-g}} \\
    K_0^G(X) \ar[u]^-{\pi_X^\ast} \ar[ur]^-{\pi_X^\ast} \ar[rr]_-{(\tau_X^G)_i} \ar[urr]_-{s_U} & & A_{i+l-g}(X \times^G U),
    }
\end{align}
where the unlabelled top right morphism is the natural isomorphism, since the diagonal $G$-action on $X \times U$ is free. $s_U$ is determined by the commutativity of the diagram and equals the composition $K_0^G(X) \xrightarrow{\pi_X^\ast} K_0^G(X \times U) \to K_0(X \times^G U)$.

\subsection{Perfect obstruction theory, virtual structure sheaf and virtual fundamental cycle} We start with defining a perfect obstruction theory.

\begin{defi} \cite{BehFan} \label{pot def}
A perfect obstruction theory on a morphism $X \to Y$, where $X$ is a \DM stack and $Y$ a smooth, pure dimensional Artin stack, is a morphism
$$\phi \colon E \lra \bL_{X/Y}$$
in $D(X)$, where $E$ is a perfect complex of amplitude $[-1,0]$, satisfying that $h^{-1}(\phi)$ is surjective and $h^0(\phi)$ is an isomorphism.

We refer to the coherent sheaf $\Ob_X := h^1 (E^\vee)$ as the obstruction sheaf associated to the perfect obstruction theory $\phi$.
\end{defi}

A perfect obstruction theory $\phi$ induces a Cartesian diagram
\begin{align} \label{loc 3.1}
    \xymatrix{
    \cC_{X/Y} \ar[r] \ar[d] & \nN_{X/Y} := h^1 /h^0 (\bL_{X/Y}^\vee) \ar[rr]^-{h^1 /h^0 (\phi^\vee)} \ar[d] && \eE := h^1 / h^0 (E^\vee) \ar[d] \\
    \fc_{X/Y} \ar[r] & \fn_{X/Y} = h^1(\bL_{X/Y}^\vee) \ar[rr]_-{h^1(\phi^\vee)} && \Ob_X= h^1 (E^\vee)
    }
\end{align}
where $\cC_{X/Y}$ and $\nN_{X/Y}$ are the intrinsic normal cone and intrinsic normal sheaf of $X$ over $Y$ respectively, while $\fc_{X/Y}$ and $\fn_{X/Y}$ are their coarse moduli sheaves. As we will recall in Section~\ref{equivariant K-theory and Gysin maps background section}, $\fn_{X/Y}$ and $\Ob_X$ can be viewed as coherent sheaves on $X$ as well as sheaf stacks over $X$. These are not algebraic in general, except in special cases, such as when the sheaves are locally free. However, they do admit a notion of local charts and are equipped with natural Gysin maps. All the horizontal arrows in the diagram are closed embeddings.

Assuming that $E$ admits a global resolution by vector bundles on $X$ so that we may write $E = [E^{-1} \to E^0]$, we have that $\eE = [E_1 / E_0]$, where we denote $E_0 = (E^0)^\vee, E_1 = (E^{-1})^\vee$. In this case, the above diagram can be augmented by the Cartesian square
\begin{align}
    \xymatrix{
    C_1 \ar[d] \ar[r] & E_1 \ar[d] \\
    \cC_{X/Y} \ar[r] & \eE = [E_1 / E_0].
    }
\end{align}

We may then define the virtual fundamental cycle $[X]\virt \in A_*(X)$ and virtual structure sheaf $[\sO_X\virt] \in K_0(X)$ using the intersection-theoretic and $K$-theoretic Gysin map as follows.

\begin{defi} \cite{BehFan, LiTian}
The virtual fundamental cycle of $X$ is defined by $[X]\virt = 0_{E_1}^! [C_1] \in A_*(X)$.
\end{defi}

\begin{defi} \cite{yplee}
The virtual structure sheaf of $X$ is defined by $[\sO_X\virt] = 0_{E_1}^! [C_1] \in K_0(X)$.
\end{defi}

\subsection{Virtual Riemann--Roch for schemes with perfect obstruction theory} In \cite{VirRR}, Fantechi and G\"{o}ttsche prove a virtual Riemann--Roch theorem for quasiprojective schemes $X$ endowed with a perfect obstruction theory $\phi \colon E \to \bL_{X}$, under the assumption that $E$ admits a global resolution $[E^{-1} \to E^0]$. To state their theorem, we define the $K$-theoretic virtual tangent bundle of $X$ by the formula
$$T_{X}\virt = [E_0] - [E_1] \in K^0(X), $$
where as above $E_0 = (E^0)^\vee, E_1 = (E^{-1})^\vee$.

\begin{thm} \cite[Theorem~3.3, Lemma~3.5]{VirRR} \label{gf virtual rr}
We have
\begin{align}
    \tau_X ( [\sO_X\virt] ) = \td (T_{X}\virt) \cap [X]\virt.
\end{align}
As a corollary, if $f \colon X \to Y$ is a proper morphism from $X$ to a smooth, pure dimensional scheme $Y$ and $V \in K^0(X)$, we obtain the virtual Grothendieck--Riemann--Roch formula
\begin{align*}
    \mathrm{ch}(f_\ast(V\otimes [\sO_X\virt])) \cdot \td(T_Y) \cap [Y] = f_\ast \left( \mathrm{ch}(V) \cdot \td (T_{X}\virt) \cap [X]\virt \right).
\end{align*}
\end{thm}

\subsection{Equivariant virtual Riemann--Roch} In \cite{EquivVirRR}, Ravi and Sreedhar prove an equivariant virtual Riemann--Roch theorem. A statement of their result is as follows.

\begin{thm}\cite[Theorem~1.1]{EquivVirRR} \label{equiv vir rr thm}
Let $X, Y$ be $G$-schemes and $f : X \to Y$ be a $G$-equivariant morphism such that $Y$ is smooth, $G$-equivariantly connected and pure dimensional. Suppose that there exists a $G$-invariant closed embedding $X \to M$ for some smooth $G$-scheme $M$ such that for any $l$-dimensional good pair for the integer $n-i$ the quotient $X \times^G U$ is a scheme. Then for any equivariant perfect relative obstruction theory $E \to \bL_{X/Y}$ on $X$ with respect to $Y$
which admits a global resolution $E = [E^{-1} \to E^0]$, we have
\begin{align}
    \tau_X ( [\sO_X\virt] ) = \td ^G(T_{X/Y}\virt) \cap [X]\virt \in A_*^G(X),
\end{align}
where $\td^G (T_{X/Y}\virt) = [E_0] - [E_1] + [f^\ast T_Y] -[\fG] \in K^0_G(X)$ and $E_0 = (E^0)^\vee, E_1 = (E^{-1})^\vee$ are $G$-equivariant locally free sheaves on $X$.
\end{thm}

\section{Equivariant $K$-Theory and Gysin Maps on Sheaf Stacks} \label{equivariant K-theory and Gysin maps background section}

Throughout this section, $G$ denotes a linear algebraic group, $X$ is a Deligne--Mumford stack with an action of $G$, $Y$ is a smooth Artin stack of pure dimension with a $G$-action, and $f \colon X \to Y$ a $G$-equivariant morphism. $\fF$ will be a $G$-equivariant coherent sheaf on $X$.

In \cite{KiemSavvas}, the authors introduced the notion of coherent sheaves on the sheaf stack $\fF$, defined the $K$-theory of $\fF$ and constructed a Gysin map $0^!_\cF$ of $\fF$, which made it possible to construct the virtual structure sheaf $[\sO_X\virt]\in K_0(X)$ for a \DM stack $X$ equipped with an almost perfect obstruction theory. This was performed in the non-equivariant case where $G$ is the trivial group.

In this section, we recall the main constructions that will be necessary from \cite{KiemSavvas} and at the same time develop their generalization to the equivariant context.

\subsection{Sheaf stacks, local charts and common roofs} We may associate a sheaf stack to a coherent sheaf on $X$.  

\begin{defi} \emph{(Sheaf stack)} The \emph{sheaf stack} associated to $\fF$ is the stack that to every morphism $\rho \colon T \to X$ from a scheme $T$ associates the set of sections $\Gamma(T, \rho^\ast \fF)$. 
\end{defi}
 
By abuse of notation, we denote by $\fF$ the sheaf stack associated to a coherent sheaf $\fF$ on $X$.  

When $\fF$ is a $G$-equivariant sheaf, the sheaf stack $\fF$ naturally admits an induced $G$-action: Let $\sigma \colon X \times G \to X$ be the $G$-action morphism on $X$ and $\pi_X \colon X \times G \to X$ the first projection map. By definition, $\fF$ is a $G$-equivariant sheaf if there exists an isomorphism
\begin{align*}
    \mu \colon \sigma^* \fF \lr \pi_X^\ast \fF.
\end{align*}
To specify the $G$-action on the associated sheaf stack, we need a morphism
\begin{align*}
    \nu \colon \fF \times G \lr \fF
\end{align*}
satisfying the usual axioms of a group action.

Consider a $T$-point of $\fF \times G$. This consists of the data
\begin{align*}
    \rho \colon T \to X,\ s \in \Gamma(T, \rho^* \fF),\ g \colon T \to G.
\end{align*}
The image of this point under $\nu$ is then the $T$-point of $\fF$ consisting of the data
\begin{align*}
    \nu(\rho) \colon T \xrightarrow{\rho \times g} X \times G \xrightarrow{\sigma} X,\ \nu(s) \in \Gamma(T, \nu(\rho)^\ast \fF),
\end{align*}
where the section $\nu(s)$ is defined as the composition
\begin{align*}
    \nu(s) \colon \oO_T \xrightarrow{s} \rho^\ast \fF = (\rho \times g)^\ast \pi_1^\ast \fF \xrightarrow{(\rho \times g)^\ast \mu^{-1}} (\rho \times g)^\ast \sigma^\ast \fF = \nu(\rho)^\ast \fF.
\end{align*}
It is routine to check that $\nu$ defines a $G$-action and moreover the projection morphism $\fF \to X$ is $G$-equivariant.

A sheaf stack is not algebraic in general and we need an appropriate notion of local charts for geometric constructions. 

\begin{defi} \emph{(Local chart)} \label{local chart}
A \emph{local chart} $Q=(U, \rho, E, r_E)$ for the sheaf stack $\fF$ consists of
\begin{enumerate}
\item an \'etale morphism $\rho:U\to X$ from a scheme $U$, and 
\item a surjective homomorphism $r_E:E\to \rho^* \fF=\fF|_U$ of coherent sheaves on $U$ from a locally free sheaf $E$ on $U$.
\end{enumerate}
We will call $U$ the \emph{base} of the chart $Q$. 
If $U$ is affine and $E$ is free, then the local chart $Q = (U, \rho, E, r_E)$ is called \emph{affine}.

When $\fF$ is $G$-equivariant, we say that the chart $Q$ is $G$-equivariant if $U$ is a $G$-scheme, $\rho$ is $G$-equivariant, $E$ is a $G$-equivariant locally free sheaf and the homomorphism $r_E$ is $G$-equivariant.
\end{defi}

\begin{defi} \emph{(Morphism between local charts)}
Let $Q=(U, \rho, E, r_E)$ and $Q'=(U', \rho', E', r_{E'})$ be two local charts for $\fF$. A \emph{morphism} $\gamma \colon Q \to Q'$ is the pair $(\rho_\gamma, r_\gamma)$ of an \'{e}tale morphism $\rho_\gamma \colon U \to U'$ and a surjection $r_\gamma \colon E \to \rho_\gamma^* E'$ of locally free sheaves, such that the diagrams
\begin{align} \label{loc 2.1}
\xymatrix{
U \ar[r]^-{\rho_\gamma} \ar[dr]_-{\rho} & U' \ar[d]^-{\rho'} \\
& X
} \ \textrm{} \ \xymatrix{
E \ar[r]^-{r_\gamma} \ar[dr]_-{r_E} & \rho_\gamma^* E' \ar[d]^-{\rho_\gamma^* r_{E'}} \\
& \fF|_U
}
\end{align}
are commutative.

We say that $Q$ is a \emph{restriction} of $Q'$ and write $Q = Q'|_U$ if $E = \rho_\gamma^* E'$ and $r_\gamma$ is the identity morphism.

For a $G$-equivariant sheaf $\fF$ and equivariant local charts $Q, Q'$, the morphism $\gamma \colon Q \to Q'$ is called $G$-equivariant when $\rho_\gamma$ and $r_\gamma$ are equivariant.
\end{defi}

The notion of a common roof enables us compare two local charts on $\fF$ with the same base $\rho \colon U \to X$.

\begin{defi} \emph{(Common roof)} \label{common roof defi}
Let $r:\cG\to \fF$ and $r':\cG'\to \fF$ be two surjective homomorphisms of coherent sheaves on a scheme $U$. Their fiber product is defined by
\beq\label{1}\cG \times_{\fF} \cG':=\mathrm{ker}\left( \cG\oplus \cG'\xrightarrow{(r,-r')} \fF\oplus \fF\xrightarrow{+} \fF \right)\eeq 
and we have a commutative diagram
\[\xymatrix{
\cG\times_{\fF}\cG'\ar[r]\ar[d] & \cG\ar[d] \\
\cG'\ar[r] & \cF
}\]
of surjective homomorphisms, which is 
universal among such diagrams of surjective homomorphisms in the obvious sense. 

Given two charts $Q=(U, \rho, E, r_E)$ and $Q'=(U, \rho, E', r_{E'})$ with the same quasi-projective base $U$, we can pick a surjective homomorphism 
\begin{align} \label{loc 2.2}
     W\lra E\times_{\fF|_U}E'
\end{align} 
from a locally free sheaf $W$. Denoting the induced surjection $W\to  \fF|_U$ by $r_W$, we obtain a local chart $(U,\rho, W, r_W)$ with natural morphisms to $Q$ and $Q'$, which we call a \emph{common roof} of $Q$ and $Q'$.

When $\fF$ and $Q, Q'$ are $G$-equivariant, we obtain an equivariant common roof by requiring that $W$ and the morphism~\eqref{loc 2.2} are $G$-equivariant.
\end{defi}

More generally, given two charts $Q=(U,\rho, E, r_E)$ and $Q'=(U',\rho',E',r_{E'})$ of the sheaf stack $\cF$, we let $V=U\times_XU'$ and have two local charts $Q|_V$ and $Q'|_V$ with the same base. A common roof of $Q|_V$ and $Q'|_V$ is called a \emph{common roof} of $Q$ and $Q'$.

\subsection{Coherent sheaves on a sheaf stack $\fF$} 

A \emph{coherent sheaf} $\aA$ on $\fF$ is an assignment to every local chart $Q = (U, \rho, E, r_E)$ of a coherent sheaf $\aA_Q$ on the scheme $E$ (in the \'{e}tale topology) such that for every morphism $\gamma \colon Q \to Q'$ between local charts there exists an isomorphism 
\beq\label{y2} r_\gamma^* \left( \rho_\gamma^*\aA_{Q'} \right) \lr \aA_Q\eeq
which satisfies the usual compatibilities for composition of morphisms. 
Note that we abusively write $\rho_\gamma^* \aA_{Q'}$ for the pullback of $\aA_{Q'}$ to $\rho_\gamma^* E'$ via the morphism of bundles $\rho_\gamma^* E' \to E'$ induced by $\rho_\gamma$.
A \emph{quasicoherent sheaf} on a sheaf stack is defined likewise. 

A homomorphism $f:\aA\to \bB$ of (quasi)coherent sheaves on $\fF$ is the data of a homomorphism
$f_Q:\aA_Q\to \bB_Q$ of (quasi)coherent sheaves on $E$ for each local chart $Q=(U,\rho,E,r_E)$ such that for every morphism $\gamma:Q\to Q'$ of local charts,
the diagram 
$$\xymatrix{
r_\gamma^* \left( \rho_\gamma^*\aA_{Q'} \right) \ar[r]\ar[d]_{f_{Q'}}& \aA_Q\ar[d]^{f_Q}\\
r_\gamma^* \left( \rho_\gamma^*\bB_{Q'} \right) \ar[r] & \bB_Q
}$$
is commutative where the horizontal arrows are as in \eqref{y2}. We say that a homomorphism $f:\aA\to \bB$ is an isomorphism if $f_Q$ is an isomorphism for each local chart $Q$. 

We would now like to define $G$-equivariant sheaves on $\fF$ when $\fF$ and $X$ admit compatible $G$-actions. In order to do so, we need the following lemma, which says that our definition of a sheaf coincides with the usual definition of a sheaf on a stack.

\begin{lem} \label{sheaf on sheaf stack is usual sheaf}
A (quasi)coherent sheaf $\aA$ on a sheaf stack $\fF$ is equivalent to an assignment to every $T$-point $\tau \colon T \to \fF$ of a (quasi)coherent sheaf $\aA_\tau$ on $T$ (in the \'{e}tale topology) such that whenever we have a composition 
$$\tau \colon T \xrightarrow{\ f \ } T' \xrightarrow{\ {\tau'} \ } \fF$$
there exists an isomorphism
\begin{align} \label{loc 2.5}
    f^\ast \aA_{\tau'} \lr \aA_{\tau}.
\end{align}
These isomorphisms satisfy the usual compatibility conditions.
\end{lem}

\begin{proof}
This is a descent statement. Clearly if we have such an assignment $\aA_\tau$ for every $T$-point $\tau \colon T \to \fF$, then for a local chart $Q = (U, \rho, E, r_E)$ we get an $E$-point of $\fF$ given by the composition
$$\tau_Q \colon E \xrightarrow{\ r_E \ } \fF|_U = \rho^\ast \fF \lr \fF,$$
where the last morphism is the natural open inclusion. 

We can thus set $\aA_Q := \aA_{\tau_Q}$. For a morphism $\gamma \colon Q \to Q'$ of local charts, we obtain an induced composition by \eqref{loc 2.1}
\begin{align} \label{loc 2.6}
    \tau_Q \colon E \xrightarrow{\ r_\gamma \ } E' \xrightarrow{\tau_{Q'}} \fF
\end{align}
where by slight abuse of notation $r_\gamma$ is the map of vector bundles fitting in the commutative diagram
\begin{align*}
    \xymatrix{
    E \ar[d] \ar[r]^-{r_\gamma} & E' \ar[d] \\
    U \ar[r]_-{\rho_\gamma} & U'.
    }
\end{align*}
The isomorphism~\eqref{loc 2.5} corresponding to the composition~\eqref{loc 2.6} can be written equivalently as
$$r_\gamma^\ast \left( \rho_\gamma^\ast \aA_{Q'} \right) \lr \aA_Q.$$
Hence we obtain a (quasi)coherent sheaf on $\fF$.

Conversely, suppose we are given a (quasi)coherent sheaf $\aA$ on $\fF$ and a $T$-point $\tau \colon T \to \fF$ which consists of the data of a morphism $\rho \colon T \to X$ and a section 
$$s \colon \oO_T \lr \rho^\ast \fF.$$

We need to specify the assignment of a sheaf $\aA_\tau$ on $T$. By the pullback formalism developed in \cite[Section 3]{KiemSavvasLoc}, there exists a functorial pullback sheaf $\rho^* \aA$ on the sheaf stack $\rho^* \fF$ over $T$. By replacing $X$ and $\fF$ with $T$ and $\rho^* \fF$ we may thus assume that $T=X$ and $\rho = \id_X$.

Since $\fF$ is a coherent sheaf on $X$, there is an affine \'{e}tale cover $\lbrace \rho\lalp \colon X\lalp \to X \rbrace$ and local charts $Q\lalp = (X\lalp, \rho\lalp, E\lalp, r\lalp)$ for $\fF$. By possibly shrinking $X\lalp$, we may assume that $E\lalp$ is free and thus the section $s\lalp := s|_{X\lalp}$ of $\fF|_{X\lalp}$ lifts to a section of $E\lalp$
\begin{align}
    \xymatrix{
     & E\lalp \ar[d]^-{r\lalp} \\
    \oO_{X\lalp} \ar[r]_-{s|_{X\lalp}} \ar[ur]^-{\tilde{s}\lalp} & \fF|_{X\lalp}.
    }
\end{align}

In particular, we obtain a morphism of schemes $\tilde{s}\lalp \colon X\lalp \to E\lalp$ and can define $\aA\lalp := (\tilde{s}\lalp)^\ast \aA_{Q\lalp}$. Observe that $\aA\lalp$ does not depend on the choice of lift: For any two choices $\tilde{s}\lalp, \tilde{s}\lalp'$, their difference $\tilde{s}\lalp'- \tilde{s}\lalp$ factors through a section $k\lalp$ of $\ker( r\lalp )$. We thus have commutative diagrams 
\begin{align}
    \xymatrix{
     & E\lalp \oplus \ker( r\lalp ) \ar[d]^-{\pi_1} \\
     & E\lalp \ar[d]^-{r\lalp} \\
    \oO_{X\lalp} \ar[r]_-{s|_{X\lalp}} \ar[ur]_-{\tilde{s}\lalp} \ar[ruu]^-{\tilde{s}\lalp \oplus k\lalp} & \fF|_{X\lalp},
    } \xymatrix{
     & E\lalp \oplus \ker( r\lalp ) \ar[d]^-{+} \\
     & E\lalp \ar[d]^-{r\lalp} \\
    \oO_{X\lalp} \ar[r]_-{s|_{X\lalp}} \ar[ur]_-{\tilde{s}\lalp'} \ar[ruu]^-{\tilde{s}\lalp \oplus k\lalp} & \fF|_{X\lalp}.
    }
\end{align}

Since $\tilde{r}\lalp := r\lalp \circ \pi_1 = r\lalp \circ +$, by taking a surjection $F\lalp \to \ker(r\lalp)$ with $F\lalp$ free, we get a local chart $\tilde{Q}\lalp = (X\lalp, \rho\lalp, E\lalp \oplus F\lalp, \tilde{r}\lalp)$ and two morphisms $\gamma, \gamma' \colon \tilde{Q}\lalp \to Q\lalp$ determined by the morphisms $E\lalp \oplus F\lalp \xrightarrow{\pi_1} E\lalp$ and $E\lalp \oplus F\lalp \xrightarrow{+} E\lalp$.

Lifting $k\lalp$ to any section $\tilde{k}\lalp$ of $F\lalp$, we get new commutative diagrams
\begin{align}
    \xymatrix{
     & E\lalp \oplus F\lalp \ar[d]^-{\pi_1} \\
     & E\lalp \ar[d]^-{r\lalp} \\
    \oO_{X\lalp} \ar[r]_-{s|_{X\lalp}} \ar[ur]_-{\tilde{s}\lalp} \ar[ruu]^-{\tilde{s}\lalp \oplus \tilde{k}\lalp} & \fF|_{X\lalp},
    } \xymatrix{
     & E\lalp \oplus F\lalp \ar[d]^-{+} \\
     & E\lalp \ar[d]^-{r\lalp} \\
    \oO_{X\lalp} \ar[r]_-{s|_{X\lalp}} \ar[ur]_-{\tilde{s}\lalp'} \ar[ruu]^-{\tilde{s}\lalp \oplus \tilde{k}\lalp} & \fF|_{X\lalp}.
    }
\end{align}

The isomorphisms~\eqref{y2} now immediately imply that there is a natural isomorphism $\pi_1^\ast \aA_{Q\lalp} \simeq +^\ast \aA_{Q\lalp}$ of sheaves on the vector bundle $E\lalp \oplus F\lalp$. Pulling back via $\tilde{s}\lalp \oplus \tilde{k}\lalp$, we see that $\aA\lalp$ is defined up to canonical isomorphism.

Using the same reasoning and common roofs for overlaps, a standard descent argument implies that the sheaves $\aA\lalp$ glue canonically to define a sheaf on $X$, which by definition is the sheaf $\aA_\tau$. The fact that $\aA_\tau$ satisfies the functoriality condition~\eqref{loc 2.5} is routine. We leave the details to the reader.
\end{proof}

We are now in position to define $G$-equivariant sheaves on sheaf stacks.

\begin{defi}
Let $X$ be a \DM stack with a $G$-action and $\fF$ a $G$-equivariant coherent sheaf on $X$, with $G$-action determined by the morphism $\nu \colon \fF \times G \to \fF$. 

We say that a (quasi)coherent sheaf $\aA$ on $\fF$ is $G$-equivariant if there exists an isomorphism
$$\nu^* \aA \lr \pi_\fF^\ast \aA$$
satisfying the axioms of a group action, where $\pi_\fF \colon \fF \times G \to \fF$ denotes the projection. Morphisms between $G$-equivariant sheaves are defined in the obvious way.
\end{defi}

By Lemma~\ref{sheaf on sheaf stack is usual sheaf} the meaning of pullback in the definition is unambiguous. Observe that, as expected, if $\aA$ is a $G$-equivariant sheaf on $\fF$ and $Q = (U, \rho, E, r_E)$ is a $G$-equivariant local chart, then $\aA_Q$ is a $G$-equivariant sheaf on the $G$-scheme $E$.

Exact sequences and the $K$-group $K_0(\fF)$ were defined in \cite{KiemSavvas} as follows.
\begin{defi} \emph{(Short exact sequence)} \label{ses}
Let $\aA, \bB, \cC$ be coherent sheaves on the sheaf stack $\fF$. A sequence 
\begin{align*}
    0 \lr \aA \lr \bB \lr \cC \lr 0
\end{align*}
of homomorphisms of coherent sheaves on $\fF$  
is \emph{exact} if for every local chart $Q=(U, \rho, E, r_E)$ on $\fF$ the sequence
\begin{align*}
    0 \lr \aA_Q \lr \bB_Q \lr \cC_Q \lr 0
\end{align*}
is an exact sequence of coherent sheaves on the scheme $E$.
\end{defi} 
Note that the morphism
$$E\mapright{r_\gamma} \rho_\gamma^*E'=U\times_{U'}E'\lra E'$$
is smooth and hence flat.  
We can likewise define the \emph{kernel} and \emph{cokernel} of a homomorphism $f:\aA\to \bB$ of coherent sheaves on the sheaf stack $\fF$. Thus coherent and quasicoherent sheaves on $\fF$  form abelian categories $$\mathrm{Coh}(\fF) \subset \mathrm{QCoh}(\fF).$$

In the $G$-equivariant case, we analogously obtain abelian categories 
$$\mathrm{Coh}^G(\fF) \subset \mathrm{QCoh}^G(\fF),$$
with objects $G$-equivariant sheaves and $G$-equivariant morphisms between them.

\begin{defi} 
The \emph{$K$-group} of coherent sheaves on $\fF$ is the group $K_0(\fF)$ generated by the isomorphism classes $[\aA]$ of coherent sheaves $\aA$ on $\fF$, with relations generated by $[\bB] = [\aA] + [\cC]$ for every short exact sequence $$0 \lr \aA \lr \bB \lr \cC \lr 0.$$

When $\fF$ admits a $G$-action, we have the corresponding notion of the \emph{$G$-equivariant $K$-group} $K_0^G(\fF)$.
\end{defi}
In other words, $K_0(\cF)$ and $K_0^G(\cF)$ are the Grothendieck groups of the abelian categories $\mathrm{Coh}(\cF)$ and $\mathrm{Coh}^G(\cF)$ respectively. 

If $\fF=E$ is locally free, so that $\fF$ is an algebraic stack, then the above definitions recover the standard notions of short exact sequences and $K_0(E)$ and $K_0^G(E)$ for the vector bundle $E$.

\subsection{Equivariant Gysin maps for sheaf stacks in $K$-theory} \label{Koszul homology section} 
Let $Q = (U, \rho, E, r_E)$ be a local chart for the sheaf stack $\fF$ and denote the vector bundle projection map $E \to U$ by $\pi_E$. The tautological section of the pullback $\pi_E^* E$ induces an associated Koszul complex $$\cK(E):=\wedge^\bullet \pi_E^* E^\vee$$ that resolves the structure sheaf $\oO_U$ of the zero section of $\pi_E$.

\begin{defi}
For any local chart $Q = (U, \rho, E, r_E)$ and a coherent sheaf $\aA$ on $\fF$, the \emph{$i$-th Koszul homology sheaf} $\hH_Q^i(\aA)$ of $\aA$ with respect to $Q$ is defined as the homology of the complex 
$$\cK(E)\otimes_{\sO_E} \aA_Q=\wedge^\bullet \pi_E^* E^\vee \otimes_{\sO_E} \aA_Q$$ 
in degree $-i$. Since the complex $\cK(E) \otimes_{\sO_E} \aA_Q$ is quasi-isomorphic to the derived tensor product $\sO_U \otimes^L_{\sO_E} \aA_Q$, $\hH_Q^i(\aA)$ is a coherent sheaf supported on the zero section of $E$ and hence we may consider it naturally as a coherent sheaf on $U$, so that $\hH_Q^i(\aA) \in \Coh(U)$.
\end{defi}

Using common roofs and the standard descent theory for coherent sheaves on algebraic stacks, the following is proven in \cite{KiemSavvas}.

\begin{thm-defi}
Let $\aA$ be a coherent sheaf on a sheaf stack $\fF$ over a Deligne--Mumford stack $X$. The coherent sheaves $\hH_Q^i(\aA)\in \mathrm{Coh}(U)$ glue canonically to a coherent sheaf $\hH_\cK^i(\aA)\in \mathrm{Coh}(X)$ on $X$, which is defined to be the $i$-th Koszul homology sheaf of $\aA$.
\end{thm-defi} 

The Koszul homology sheaves of $\aA$ can then be used to define a $K$-theoretic Gysin map $0_\fF^! \colon K_0(\fF) \to K_0(X)$.

\begin{defi}\label{32} \emph{($K$-theoretic Gysin map)}
The $K$-theoretic Gysin map  is defined by the formula
\begin{align}\label{31}
 0_\fF^! \colon K_0(\fF) \to K_0(X),\quad   0_{\fF}^! [\aA] = \sum_{i \geq 0} (-1)^i [ \hH_\cK^i(\aA) ] \in K_0(X)
\end{align}
where $\aA$ is a coherent sheaf on $\fF$.
\end{defi}

For the purposes of this paper, we need equivariant generalizations of these constructions. These are provided by the following theorem when $X$ is an algebraic space.

\begin{thm} \label{equivariant gysin map k-theory alg spc}
Suppose that $\fF$ is a $G$-equivariant sheaf over an algebraic space $X$ with a $G$-action. For any $G$-equivariant coherent sheaf $\aA$ on the sheaf stack $\fF$, the Koszul homology sheaves $\hH_\kK^i(\aA)$ admit natural equivariant lifts in $\Coh^G(X)$ and there exists a $G$-equivariant $K$-theoretic Gysin map
\begin{align} \label{equivariant gysin map k-theory}
    0_\fF^{!,G} \colon K_0^G(\fF) \to K_0^G(X)
\end{align}
defined by the formula~\eqref{31}.
\end{thm}

\begin{proof}
Let $\aA$ be a $G$-equivariant sheaf on the sheaf stack $\fF$. Let $U$ be a quasi-affine scheme on which $G$ acts freely and write $\pi_X \colon X \times U \to X$ for the projection map. The diagonal action of $G$ on $X \times U$ is free and thus by Theorem~\ref{theorem 2.3}(3), $X \times U$ admits an affine $G$-equivariant atlas $V \to X \times U$. In particular, we obtain a smooth affine $G$-equivariant surjection $\rho \colon V \to X$.

Consider the commutative diagrams with Cartesian rhombi
\begin{align}
    \xymatrix{
    & \rho^\ast \fF \times_{\fF} \rho^\ast \fF \ar[dl]_{\pi_1} \ar[dr]^{\pi_2} \ar[dd] & \\
    \rho^\ast \fF \ar[dd] & & \rho^\ast \fF \ar[dd] \\
    & V \times_X V \ar[dl]_{\pi_1} \ar[dr]^{\pi_2} & \\
    V \ar[dr]_-{\rho} & & V \ar[dl]^-{\rho} \\
    & X &
    }
    \xymatrix{
    & \rho^\ast \fF \times_{\fF} \rho^\ast \fF \ar[dl]_{\pi_1} \ar[dr]^{\pi_2} & \\
    \rho^\ast \fF \ar[dd] \ar[dr]^-{\rho} & & \rho^\ast \fF \ar[dd] \ar[dl]_-{\rho}\\
    & \fF \ar[dd] & \\
    V \ar[dr]_-{\rho} & & V \ar[dl]^-{\rho} \\
    & X &
    }
\end{align}

By smooth descent, we may equivalently think of $\aA$ as a $G$-equivariant sheaf $\rho^\ast \aA$ on the sheaf stack $\rho^\ast \fF$ together with a $G$-equivariant descent datum $\pi_1^\ast \rho^\ast \aA \simeq \pi_2^\ast \rho^\ast \aA$, an isomorphism between the pullbacks of $\rho^\ast \aA$ via the two projections $\pi_1, \pi_2$ to the sheaf stack $\rho^\ast \fF \times_{\fF} \rho^\ast \fF$.

Since $V$ satisfies the resolution property, there exists a $G$-equivariant locally free sheaf $E$ on $V$ and a surjection $r \colon E \to \rho^\ast \fF$, so that we have a surjective local chart $Q_V = (V, \id_V, E, r)$ for $\rho^\ast \fF$ on $V$. We then obtain $G$-equivariant Koszul homology sheaves $\hH_\kK^i(\rho^\ast \aA)$ on $V$ defined as the homology sheaves of the $G$-equivariant complex
$$\kK(E) \otimes_{\sO_E} (\rho^\ast \aA)_{Q_V}.$$

Pulling back, we obtain surjective local charts $Q_i = (V \times_X V, \id, \pi_i^\ast E, \pi_i^\ast r)$ for the sheaf stack $\pi_i^\ast \rho^\ast \fF \simeq \rho^\ast \fF \times_{\fF} \rho^\ast \fF$ where $i=1,2$. Since $X$ is separated and $V$ can be taken to be affine, $V \times_X V$ is also an affine $G$-scheme, which by Theorem~\ref{theorem 2.3}(1) satisfies the resolution property. In particular, there exists a $G$-equivariant surjection
$$W \lr \pi_1^\ast E \times_{\rho^\ast \fF \times_{\fF} \rho^\ast \fF} \pi_2^\ast E$$
where $W$ is a $G$-equivariant locally free sheaf. By Definition~\ref{common roof defi}, this gives a local chart $Q_W = (V \times_X V, \id, W, r_W)$ which is a $G$-equivariant common roof for $Q_1$ and $Q_2$.

Using this common roof and the canonical nature of the Koszul homology sheaves, the descent datum $\pi_1^\ast \rho^\ast \aA \simeq \pi_2^\ast \rho^\ast \aA$ induces a $G$-equivariant isomorphism of coherent sheaves on $V$
\begin{align} \label{loc 3.14}
    \pi_1^\ast \hH_\kK^i(\rho^\ast \aA) \simeq \hH_\kK^i (\pi_1^\ast \rho^\ast \aA) \simeq \hH_\kK^i (\pi_2^\ast \rho^\ast \aA) \simeq \pi_2^\ast \hH_\kK^i(\rho^\ast \aA).
\end{align}

It follows by smooth descent that the Koszul homology sheaves $\hH_\kK^i(\rho^\ast \aA)$ descend to $G$-equivariant sheaves $\hH_\kK^i(\aA)$ on $X$. These are independent of the choices of smooth affine equivariant atlas $\rho \colon V \to X$ and surjection $E \to \rho^\ast \fF$, since for any other choices $\rho' \colon V' \to X$ and $E' \to (\rho')^\ast \fF$, we may work with the affine equivariant atlas 
\begin{align} \label{loc 3.12}
    V \times_X V' \to X
\end{align} 
and any equivariant surjection 
\begin{align} \label{loc 3.13}
    W' \to \pi_1^\ast E \times_{\rho^\ast \fF \times_{\fF} (\rho')^\ast \fF} \pi_2^\ast E'
\end{align}
from an equivariant locally free sheaf $W'$, where $\pi_1 \colon V \times_X V' \to V$ and $\pi_2 \colon V\times_X V' \to V'$ are the two projection maps. It is routine to check that the $G$-equivariant sheaves $\hH_\kK^i(\aA)$ are the same as the ones obtained using the atlas~\eqref{loc 3.12} and surjection from a locally free sheaf~\eqref{loc 3.13}, so they are independent of any choices made.

Forgetting the equivariant structure, as coherent sheaves on $X$, these are isomorphic to the previously defined Koszul homology sheaves $\hH_\kK^i(\aA)$. This is because for any local chart for $\fF$ we obtain by pullback an induced local chart for $\rho^\ast \fF$. By the definition and canonicity of the Koszul homology sheaves, it is then immediate that the Koszul homology sheaves $\hH_\kK^i (\rho^\ast \aA)$ are alternatively obtained using a cover by these induced local charts and the corresponding descent datum is compatible with~\eqref{loc 3.14} (in fact it is its pullback to the \'{e}tale cover of $V$ given by the induced local charts for $\rho^\ast \fF$). This observation concludes the proof.
\end{proof}

When $X$ is Deligne--Mumford, we make the extra assumption of the existence of an affine $G$-equivariant atlas for $\fF$. 

\begin{thm} \label{equivariant gysin map k-theory dm stack}
Suppose that $X$ is a \DM stack with a $G$-action and an affine $G$-equivariant atlas for a $G$-equivariant sheaf $\fF$. Then the conclusions of Theorem~\ref{equivariant gysin map k-theory alg spc} hold for any $G$-equivariant coherent sheaf $\aA$ on the sheaf stack $\fF$.
\end{thm}

\begin{proof}
The existence of a $G$-equivariant affine atlas for $\fF$ gives a $G$-equivariant surjective affine \'{e}tale morphism $\rho \colon V \to X$, where $V$ is a $G$-scheme. Since \DM stacks in this paper are assumed to be separated, the proof of Theorem~\ref{equivariant gysin map k-theory alg spc} applies verbatim to imply the statement. 
\end{proof}

\subsection{Equivariant Gysin maps for sheaf stacks in intersection theory}
Chang-Li \cite{LiChang} also define an intersection-theoretic Gysin map in an analogous fashion, which we denote by the same notation
\begin{align} \label{chow gysin map for sheaf stack}
    0_\fF^! \colon A_*(\fF) \lr A_*(X).
\end{align}

Here $A_*(\fF)$ are the Chow groups of the sheaf stack $\fF$, defined in \cite{LiChang}, also using local charts (cf. Definition~\ref{local chart}). This has a natural equivariant lift 
\begin{align} \label{equivariant gysin map chow}
     0_\fF^{!,G} \colon A_\ast^G(\fF) \lr A_\ast^G(X).
\end{align}

We remark that we will only be applying the equivariant Gysin map to classes in $A_*^G (\fF)$ represented by linear combinations of $G$-invariant integral cycles $Z \subseteq \fF$. The general definition is obtained by the same procedure described in Section~\ref{RR for quotient sec}, considering the Chow groups
$$A_i^G(\fF) := A_{i+l-g} (\fF \times^G U) $$
for $l$-dimensional good pairs $(V,U)$ for the integer $n-i$, where $\fF \times^G U$ denotes the descent to $X \times^G U$ of the pullback of $\fF$ to the product $X \times U$, and applying the usual Gysin map $0_{\fF \times^G U}^!$.

\section{Almost Perfect Obstruction Theory, Virtual Structure Sheaf and Virtual Fundamental Cycle}\label{APOT section}

In this short section, we recall the definitions of an almost perfect obstruction given in \cite{KiemSavvas} and the associated virtual structure sheaf and fundamental cycle and then generalize them to the equivariant setting.

\begin{defi} \emph{(Almost perfect obstruction theory)} \label{APOT}
Let $X \to Y$ be a morphism, where $X$ is a Deligne--Mumford stack of finite presentation and $Y$ is a smooth Artin stack of pure dimension. An \emph{almost perfect obstruction theory} $\phi$ consists of an \'{e}tale covering $\lbrace X_\alpha \to X \rbrace_{\alpha \in A}$ 
of $X$ and perfect obstruction theories $\phi_\alpha \colon E_\alpha \to \bL_{X_\alpha / Y}$ of $X_\alpha$ over $Y$ such that the following hold. 
\begin{enumerate}
\item For each pair of indices $\alpha, \beta$, there exists an isomorphism \begin{align*}
\psi_{\alpha \beta} \colon \Ob_{X_\alpha} \vert_{X_{\alpha\beta}} \lra \Ob_{X_\beta} \vert_{X_{\alpha\beta}}
\end{align*}
so that the collection $\lbrace \Ob_{X\lalp}=h^1(E_\alpha^\vee), \psi\lab \rbrace$ gives descent data of a sheaf $\Ob_X$, called the obstruction sheaf, on $X$.
\item For each pair of indices $\alpha, \beta$, there exists an \'{e}tale covering $\lbrace V_\lambda \to X\lab \rbrace_{\lambda \in \Gamma}$ of $X\lab=X_\alpha\times_XX_\beta$ such that for any $\lambda$, the perfect obstruction theories $E_\alpha \vert_{V_{\lambda}}$ and $E_\beta \vert_{V_{\lambda}}$ are isomorphic and compatible with $\psi\lab$. This means that there exists an isomorphism
\begin{align*}
    \eta_{\alpha\beta\lambda} \colon E_\alpha \vert_{V_{\lambda}} \lr E_\beta \vert_{V_{\lambda}}
\end{align*} 
in $D^b(\Coh V_\lambda)$ fitting in a commutative diagram
\begin{align} \label{compatibility data APOT}
    \xymatrix{
    E_\alpha \vert_{V_{\lambda}} \ar[d]_-{\phi\lalp|_{V_\lambda}} \ar[r]^-{ \eta_{\alpha\beta\lambda}} & E_\beta \vert_{V_{\lambda}} \ar[d]^-{\phi\lbet|_{V_\lambda}} \\
    \bL_{X_\alpha / Y}|_{V_\lambda} \ar[r] \ar[dr] & \bL_{X_\beta / Y}|_{V_\lambda} \ar[d] \\
    & \bL_{V_\lambda /Y}
    }
\end{align}
which moreover satisfies $h^1(\eta_{\alpha\beta\lambda}^\vee) = \psi\lab^{-1}|_{V_\lambda}$.
\end{enumerate}
\end{defi}

Suppose that the morphism $X \to Y$ admits an almost perfect obstruction theory. Then the definition implies that the closed embeddings given in diagram \eqref{loc 3.1}
$$h^1(\phi\lalp^\vee) \colon \fn_{U\lalp/Y} \lr \Ob_{X\lalp}$$
glue to a global closed embedding
$$j_\phi \colon \fn_{X/Y} \lr \Ob_{X}$$
of sheaf stacks over $X$.
Therefore, the coarse intrinsic normal cone stack $\fc_{X/Y}$ embeds as a closed substack into the sheaf stack $\Ob_X$.

\begin{defi} \cite{KiemSavvas} \emph{(Virtual structure sheaf)} \label{virtual structure sheaf}
Let $X \to Y$ be as above, together with an almost perfect obstruction theory $\phi \colon X \to Y$. The \emph{virtual structure sheaf} of $X$ associated to $\phi$ is defined as
\begin{align*}
    [\oO_X\virt] := 0_{\Ob_X}^! [\oO_{\fc_{X/Y}}] \in K_0(X).
\end{align*}
\end{defi}

It is not hard to show that an almost perfect obstruction theory satisfies the axioms of a semi-perfect obstruction theory and thus by \cite{LiChang} induces a virtual fundamental cycle $[X]\virt \in A_*(X)$. This can be expressed directly using the coarse intrinsic normal cone stack $\fc_{X/ Y}$.

\begin{defi} \cite{LiChang, KiemSavvas} \emph{(Virtual fundamental cycle)} \label{virtual fundamental cycle}
Let $X \to Y$ be as above, together with an almost perfect obstruction theory $\phi \colon X \to Y$. The \emph{virtual fundamental cycle} of $X$ associated to $\phi$ is defined as
\begin{align*}
    [X]\virt := 0_{\Ob_X}^! [\fc_{X/Y}] \in A_*(X),
\end{align*}
where $0_{\Ob_X}^!$ denotes the Gysin map~\eqref{chow gysin map for sheaf stack} and $[\fc_{X/Y}] \in A_*(\Ob_X)$ is the cycle associated to the closed substack $\fc_{X/Y}$ of $\Ob_X$ (cf. \cite{LiChang}).
\end{defi}

The above definition and constructions admit direct analogues in the equivariant context. We now briefly explain the necessary adjustments.

As usual, $G$ denotes a linear algebraic group, $X$ is a Deligne--Mumford stack with an action of $G$, $Y$ is a smooth Artin stack of pure dimension with a $G$-action, and $f \colon X \to Y$ a $G$-equivariant morphism.

When $G = \bC^\ast$ and $Y$ is a point, the notion of $G$-equivariant almost perfect obstruction theory was defined in \cite[Definition~5.1]{KiemSavvasLoc} in order to develop a virtual torus localization formula. It is clear how to generalize the definition of an almost perfect obstruction theory to our setting as follows.

\begin{defi} \emph{($G$-equivariant almost perfect obstruction theory)} \label{equivariant APOT}
Let $f \colon X \to Y$ be a $G$-equivariant morphism as above. A \emph{$G$-equivariant almost perfect obstruction theory} $\phi$ consists of the following data:
\begin{enumerate}
    \item[(a)] A $G$-equivariant \'{e}tale covering $\lbrace X_\alpha \to X \rbrace_{\alpha \in A}$ of $X$.
    \item[(b)] For each index $\alpha \in A$, an object $E\lalp \in D^G(X\lalp)$ and a morphism $\phi_\alpha \colon E_\alpha \to \bL_{X_\alpha/Y}$ in $D^G(X\lalp)$ which gives a $G$-equivariant perfect obstruction theory on $X\lalp$ over $Y$.
\end{enumerate}
These are required to satisfy the following conditions:
\begin{enumerate}
\item For each pair of indices $\alpha, \beta$, there exists a $G$-equivariant isomorphism 
\begin{align*}
\psi_{\alpha \beta} \colon \Ob_{X_\alpha} \vert_{X_{\alpha\beta}} \lra \Ob_{X_\beta} \vert_{X_{\alpha\beta}}
\end{align*}
so that the collection $\lbrace \Ob_{X\lalp}=h^1(E_\alpha^\vee), \psi\lab \rbrace$ gives descent data of a $G$-equivariant coherent sheaf $\Ob_X$, called the obstruction sheaf, on $X$.
\item For each pair of indices $\alpha, \beta$, there exists a $G$-equivariant \'{e}tale covering $\lbrace V_\lambda \to X\lab \rbrace_{\lambda \in \Gamma}$ of $X\lab=X_\alpha\times_XX_\beta$ such that for any $\lambda$, the perfect obstruction theories $\phi_\alpha \vert_{V_{\lambda}}$ and $\phi_\beta \vert_{V_{\lambda}}$ are isomorphic and compatible with $\psi\lab$. This means that there exists an isomorphism
\begin{align*}
    \eta_{\alpha\beta\lambda} \colon E_\alpha \vert_{V_{\lambda}} \lr E_\beta \vert_{V_{\lambda}}
\end{align*} 
in $D^G(V\lal)$ fitting in a commutative diagram
\begin{align} \label{loc 4.1}
    \xymatrix{
    E_\alpha \vert_{V_{\lambda}} \ar[d]_-{\phi\lalp|_{V_\lambda}} \ar[r]^-{ \eta_{\alpha\beta\lambda}} & E_\beta \vert_{V_{\lambda}} \ar[d]^-{\phi\lbet|_{V_\lambda}} \\
    \bL_{X_\alpha/Y}|_{V_\lambda} \ar[r] \ar[dr] & \bL_{X_\beta/Y}|_{V_\lambda} \ar[d] \\
    & \bL_{V_\lambda/Y}
    }
\end{align}
which moreover satisfies $h^1(\eta_{\alpha\beta\lambda}^\vee) = \psi\lab^{-1}|_{V_\lambda}$.
\end{enumerate}
\end{defi}

In the above, $D^G(X\lalp)$ and $D^G(V\lal)$ denote the bounded derived categories of $T$-equivariant quasi-coherent sheaves on $U\lalp$ and $V\lal$ respectively.

By definition, the obstruction sheaf $\Ob_X$ is $G$-equivariant and the closed embedding $j_\phi \colon \fn_{X/Y} \to \Ob_X$ is $G$-invariant. Hence, $\fc_{X/Y}$ gives a $G$-equivariant sheaf $\sO_{\fc_{X/Y}}$ on the sheaf stack $\Ob_X$ and a $G$-invariant cycle $[\fc_{X/Y}]\in A_\ast^G(\Ob_X)$. We may thus give the following definition using the equivariant Gysin maps~\eqref{equivariant gysin map k-theory} and \eqref{equivariant gysin map chow}.

\begin{defi} \label{equivariant vir str sheaf and cycle def}
The $G$-equivariant virtual structure sheaf and virtual fundamental cycle are defined as
\begin{align*}
    [\sO_X\virt] & := 0_{\Ob_X}^{!,G} [\sO_{\fc_{X/Y}}] \in K_0^G(X), \\
    [X]\virt & := 0_{\Ob_X}^{!,G} [\fc_{X/Y}] \in A_*^G(X).
\end{align*}
\end{defi} 

\section{Virtual Riemann--Roch for Almost Perfect Obstruction Theories} \label{RR APOT Section}

In this section, we prove the main results of the paper, which are virtual Riemann--Roch theorems for morphisms equipped with an almost perfect obstruction theory in the non-equivariant and equivariant cases, which remove several technical assumptions and generalize the virtual Riemann--Roch theorems of Fantechi--G\"{o}ttsche \cite{VirRR} and Ravi--Sreedhar \cite{EquivVirRR}. In particular, we show that a virtual Todd class is always defined and takes the expected form in the presence of an appropriate resolution of the almost perfect obstruction theory.

Despite the fact that the usual Riemann--Roch theorem is the special case of the equivariant Riemann--Roch theorem when the group is trivial, it will be useful to develop the non-equivariant case first, as it is of independent interest and will be the basis used to establish the theorem in the equivariant case.

\subsection{The non-equivariant case} 

Throughout this subsection, $f \colon X \to Y$ denotes a morphism from an algebraic space $X$ to a smooth Artin stack $Y$ of pure dimension equipped with an almost perfect obstruction theory $\phi$ inducing a closed embedding $j_\phi \colon \fn_{X/Y} \to \Ob_X$.

\begin{defi}
The \emph{virtual Todd class} of $X$ over $Y$ is the class
\begin{align}
    \td\virt(X/Y) := \tau_X ( [\sO_X\virt] ) \in A_*(X),
\end{align}
where $\tau_X$ is the Riemann--Roch transformation~\eqref{rr transform} of $X$.
\end{defi}

Observe that the virtual Todd class only depends on the embedding $j_\phi$. The significance of the definition is in that it gives the natural substitute in the setting of almost perfect obstruction theory for the expression $\td(T_{X/Y}\virt) \cap [X]\virt$ that appears in the virtual Riemann--Roch theorem, even though the $K$-theoretic virtual tangent bundle $T_{X/Y}\virt$ does not make sense. 

As a consequence of the properties of the Riemann--Roch transformation $\tau_X$, we obtain a virtual Riemann--Roch formula as follows.

\begin{thm}
When $f$ is proper and $Y$ is a smooth scheme, we have for any $V \in K^0(X)$
\begin{align*}
    \mathrm{ch}(f_\ast(V\otimes [\sO_X\virt])) \cdot \td(T_Y) \cap [Y] = f_\ast \left( \mathrm{ch}(V) \cdot \td\virt(X/Y) \right).
\end{align*}
\end{thm}

\begin{proof}
The proof is formal using the properties of $\tau_X$ identically as in \cite[Lemma~3.5]{VirRR}.
\end{proof}

In the rest of this subsection, we explore the virtual Todd class in more detail. Our motivation is to provide an explicit formula in terms of vector bundles on $X$.

In the case that the embedding $j_\phi$ is induced by a perfect obstruction theory $\psi$, then $\td\virt(X/Y)$ recovers the usual Riemann--Roch term, as shown by the next proposition. $\psi$ does not need to be related to $\phi$ in any other way beyond inducing the embedding $j_\phi$.

\begin{prop} \label{prop 4.2}
Suppose that $\psi \colon E \to \bL_{X/Y}$ is a perfect obstruction theory on $f \colon X \to Y$ with obstruction sheaf $\Ob_X$ such that $j_\psi = h^1(\psi^\vee) = j_\phi$ and $E=[E^{-1} \to E^0]$ is a global resolution by a complex of locally free sheaves. Then
$$\td\virt(X/Y) = \frac{\td(E_0)}{\td(E_1)} \td(f^\ast T_Y) \cap [X]\virt = \td(E^\vee) \td(f^\ast T_Y) \cap [X]\virt,$$
where $E_0 = (E^0)^\vee, E_1 = (E^{-1})^\vee$.
\end{prop}

\begin{proof}
We first remark that the virtual cycles and virtual structure sheaves induced by the almost perfect obstruction theory $\phi$ and the perfect obstruction theory $\psi$ are equal, since by definition they are both given by the expression $0_{E_1}^![C_1] \in A_*(X)$ and $0_{E_1}^![\oO_{C_1}] \in K_0(X)$, where $C_1 = \fc_{X/Y} \times_{\Ob_X} E_1$. We may thus write $[X]\virt$ and $[\sO_X\virt]$ unambiguously.

Writing $T_{X/Y}\virt = [E_0] - [E_1] + [f^\ast T_Y]$, we have by applying Theorem~\ref{equiv vir rr thm} with $G$ being the trivial group that
$$\td\virt(X/Y) = \tau_X([\sO_X\virt]) = \td(T_{X/Y}\virt) \cap [X]\virt = \td(E^\vee) \td(f^\ast T_Y) \cap [X]\virt.$$
\end{proof}

More generally, suppose that the $K$-theory class $[T_{X/Y}] - [\Ob_X] \in K_0(X)$ is contained in the image of the natural map
$$\kappa \colon K^0(X) \lr K_0(X)$$
so that $[T_{X/Y}] - [\Ob_X] = \kappa(F)$ for some $F \in K^0(X)$.

We can then formulate the following conjecture.

\begin{conj} \label{conj 5.4}
$\td\virt(X/Y) = \td(F) \cap [X]\virt$ for any $F \in K^0(X)$ satisfying $\kappa (F) = [T_{X/Y}] - [\Ob_X]$.
\end{conj}

Proposition~\ref{prop 4.2} can be rephrased as saying that the conjecture holds for $F = E^\vee = [E_0] - [E_1]$ when $E = [E^{-1} \to E^0]$ is a perfect obstruction theory compatible with the embedding $j_\phi$.

We now show that the conjecture is true under certain assumptions which give a presentation of the $K$-theory class $[T_{X/Y}] - [\Ob_X] = [E_0] - [E_1]$ as an element of $K^0(X)$, but are weaker than the existence of a perfect obstruction theory inducing the closed embedding $j_\phi$.

\begin{defi} \label{def of presentation of virtual tangent} Let $E = [E^{-1} \xrightarrow{d} E^0]$ be a complex of locally free sheaves on $X$ satisfying:
\begin{enumerate}
    \item $h^0(E) = \Omega_{X/Y}$ and $h^1(E^\vee)=\Ob_X$;
    \item there exists a morphism $E^0 \to \bL_{X/Y}$ such that the natural morphism $E^0 \to h^0(E) = \Omega_{X/Y}$ factors through $\bL_{X/Y} \to h^0 (\bL_{X/Y}) = \Omega_{X/Y}$ .
\end{enumerate}
The complex $E^\vee = [E_0 \to E_1]$ is called a \emph{global presentation of the virtual tangent bundle} of $f$. We then set $T_{X/Y}\virt = [E_0] - [E_1] +[f^\ast(T_Y)] \in K^0(X)$.
\end{defi}

It is clear that $\kappa(T_{X/Y}\virt) = [T_{X/Y}] - [\Ob_X] + [f^\ast(T_Y)] \in K_0(X)$ for any global presentation of the virtual tangent bundle of $\phi$.

\begin{thm} \label{virtual todd thm}
Suppose that $X$ admits a closed embedding into a smooth algebraic space and let $E^\vee = [E_0 \to E_1]$ be a global presentation of the virtual tangent bundle of $f$. Then 
$$\td\virt(X/Y) = \frac{\td(E_0)}{\td(E_1)} \td(f^\ast T_Y) \cap [X]\virt = \td(T_{X/Y}\virt) \cap [X]\virt.$$
\end{thm}

Before proving the theorem, we need the following auxiliary lemma.

\begin{lem} \label{auxiliary lemma}
Let $E^\vee = [E_0 \to E_1]$ be a global presentation of the virtual tangent bundle of $f$. Then there exists an \'{e}tale cover $\lbrace X\lalp \to X \rbrace$ and perfect obstruction theories $\psi\lalp \colon E|_{X\lalp} \to \bL_{X\lalp/Y}$ such that $h^1(\psi\lalp^\vee) = j_\phi|_{X\lalp}$.
\end{lem}

\begin{proof}
By condition (2) in Definition~\ref{def of presentation of virtual tangent}, we have a presentation 
$$\bL_{X/Y} = [\fF^{-1} \to E^0]$$ 
for some coherent sheaf $\fF^{-1}$ on $X$.

By the existence of $\phi$, there exists an \'{e}tale cover $\lbrace X\lalp \to X \rbrace$ and perfect obstruction theories $\phi\lalp \colon E\lalp \to \bL_{X\lalp/Y}$ such that $h^1(\phi\lalp^\vee) = j_\phi|_{X\lalp}$. It thus suffices (up to possibly shrinking $X\lalp$) to construct isomorphisms $E|_{X\lalp} \cong E\lalp$ inducing the identity map on $h^1$.

Shrinking $X\lalp$ around a point $x \in X\lalp$, we may assume that $X\lalp$ is affine, $E\lalp = [G^{-1} \xrightarrow{h} G^0]$ is a complex of locally free sheaves which is minimal at $x$, meaning that $h|_x = 0$, and $\phi\lalp$ is given by a morphism of complexes
\begin{align*}
    \xymatrix{
    G^{-1} \ar[d] \ar[r] & G^0 \ar[d] \\
    \fF^{-1} \ar[r] & E^0.
    }
\end{align*}
$G_1 = (G^{-1})^\vee$ surjects onto $h^1(E\lalp^\vee) = \Ob_X|_{X\lalp}$ and the same is true for $E_1|_{X\lalp}$, so we obtain a morphism $f \colon E_1|_{X\lalp} \to G_1$ fitting into a commutative diagram
\begin{align} \label{loc 4.2}
    \xymatrix{
    E_0|_{X\lalp} \ar[r] & E_1|_{X\lalp} \ar[d]^-{f} \ar[r] & \Ob_X|_{X\lalp}. \\
    G_0 \ar[r]_-{h} & G_1 \ar[ur] &
    }
\end{align}
Since $E\lalp$ is minimal at $x$, we have that $f|_x \colon E_1|_x \to G_1|_x \cong \Ob_X|_x$ is surjective, so, up to further possible shrinking, we may assume that $f$ is surjective and we can split $E_1|_{X\lalp} = G_1 \oplus K_1$ such that $f$ is the projection onto the first factor.

By the exactness of the top row in \eqref{loc 4.2} at its middle term, the composition $E_0|_{X\lalp} \to G_1 \oplus K_1 \xrightarrow{\pi_2} K_1$ is surjective and thus we may also split $E_0|_{X\lalp} = G_0' \oplus K_1$ to obtain
\begin{align}
    \xymatrix{
    H_0 \oplus K_1 \ar[r]^-{d} & G_1 \oplus K_1 \ar[d]^-{\pi_1} \ar[r] & \Ob_X|_{X\lalp} \\
    G_0 \ar[r]_-{h} & G_1 \ar[ur] &
    }
\end{align}
where $d = d' \oplus \id_{K_1}$.

Dualizing, we get
\begin{align} \label{loc 4.4}
    \xymatrix{
    G^{-1} \oplus K^{-1} \ar[r] & H^0 \oplus K^{-1} \\
    G^{-1} \ar[u]_-{\mathrm{inc_1}} \ar[r] & G^0.
    }
\end{align}
The composition $K^{-1} \xrightarrow{\mathrm{inc}_2} H^0 \oplus K^{-1} \to \Omega_{X/Y}|_{X\lalp}$ is then zero and thus we must have a surjection $H^0 \to \Omega_{X/Y}|_{X\lalp}$. The same is true for $G^0 \to \Omega_{X/Y}|_{X\lalp}$ and hence there exists a morphism $g \colon H^0 \to G^0$ fitting in a commutative diagram
\begin{align}
    \xymatrix{
    H^0 \ar[r] \ar[d]_-{g} & \Omega_{X/Y}|_{X\lalp} \\
    G^0 \ar[ur]
    }
\end{align}

By minimality, $G^0|_x \cong \Omega_{X/Y}|_x$ and hence we may assume that $g$ is a surjection. But $H^0$ and $G^0$ have the same rank since $H^0|_x$ and $G^0|_x$ have equal dimensions, so $g$ must be an isomorphism around $x$.

It is now clear that there exists an isomorphism $k \colon G^0 \to H^0$ fitting into the commutative diagram \eqref{loc 4.4} as follows:
\begin{align} 
    \xymatrix{
    G^{-1} \oplus K^{-1} \ar[r] & H^0 \oplus K^{-1} \\
    G^{-1} \ar[u]_-{\mathrm{inc_1}} \ar[r] & G^0 \ar[u]^-{k \oplus 0}.
    }
\end{align}
This gives the required isomorphism of complexes and we are done.
\end{proof}

\begin{rmk}
While we won't explicitly use this, the lemma actually proves that the perfect obstruction theories $\phi\lalp$ forming part of the data of $\phi$ can, up to possible shrinking, be written in the form $E|_{X\lalp} \to \bL_{X\lalp/Y}$.
\end{rmk}

We can now proceed to prove Theorem~\ref{virtual todd thm}.

\begin{proof}[Proof of Theorem~\ref{virtual todd thm}]
By the definition of $E$, we have a presentation $\bL_{X/Y} = [\fF^{-1} \to E^0]$ for some coherent sheaf $\fF^{-1}$ on $X$. Write $\fF_1 = (\fF^{-1})^\vee$ for the (underived) dual of $\fF^{-1}$, so that 
$$\fn_{X/Y} = \coker(E_0 \to \fF_1) = h^1( \bT_{X/Y}).$$

We then have the surjection $r \colon E_1 \to h^1(E^\vee) = \Ob_X$, which factors as a composition $E_1 \to [E_1 / E_0] \to \Ob_X$ and fits into a diagram with Cartesian squares
\begin{align} \label{loc 4.7}
    \xymatrix{
    C_1 \ar[d] \ar[r] & M_1 \ar[d] \ar[r]^-\iota & E_1 \ar[d]^-{r} \\
    \fc_{X/Y} \ar[r] & \fn_{X/Y} \ar[r]_-{j_\phi} & \Ob_X.}
\end{align}

Write $N_1 = \Spec_{\oO_X} ( \Sym \fF^{-1} )$. We claim that $M_1 = N_1$. 

By Lemma~\ref{auxiliary lemma}, we have an \'{e}tale cover $\lbrace X\lalp \to X \rbrace$ and perfect obstruction theories $\psi\lalp \colon E|_{X\lalp} \to \bL_{X\lalp/Y}$ such that $h^1(\psi\lalp^\vee) = j_\phi|_{X\lalp}$.

Restricting to each $X\lalp$, the existence of $\psi\lalp$ locally enhances diagram~\eqref{loc 4.7} to a diagram with Cartesian squares
\begin{align*}
    \xymatrix@C+2.3pc{
    E_0|_{X\lalp} \ar[d] \ar[r] & C_1|_{X\lalp} \ar[d] \ar[r] & N_1|_{X\lalp} \ar[d] \ar[r]^-{\iota\lalp} & E_1|_{X\lalp} \ar[d] \\
    X\lalp \ar[r]^-{0_{\cC_{X/Y}}|_{X\lalp}} & \cC_{X/Y}|_{X\lalp} \ar[d] \ar[r] & \nN_{X/Y}|_{X\lalp} \ar[d] \ar[r]^-{h^1/h^0(\psi\lalp^\vee)} & [E_1|_{X\lalp} / E_0|_{X\lalp}] \ar[d] \\
    & \fc_{X/Y}|_{X\lalp} \ar[r] & \fn_{X/Y}|_{X\lalp} \ar[r]_-{h^1(\psi\lalp^\vee)=j_\phi|_{X\lalp}} & \Ob_X|_{X\lalp},}
\end{align*}
which immediately implies that the closed embeddings $\iota\lalp \colon N_1|_{X\lalp} \to E_1|_{X\lalp}$ glue to a closed embedding $N_1 \to E_1$ giving the desired identification $M_1 = N_1$ as substacks of $E_1$. 

The distinguished triangle 
\begin{align*}
    f^\ast \Omega_Y \lr \bL_X \lr \bL_{X/Y} \lr f^\ast \Omega_Y[1]
\end{align*}
gives rise to a diagram with Cartesian squares
\begin{align} \label{loc 4.8}
    \xymatrix{
     E_0 \ar[r] \ar[d] & F_0 \ar[d] \ar[r] & C_1 \ar[d] \ar[r] & N_1 \ar[d] \\
     X \ar[r] & f^\ast T_Y \ar[d] \ar[r] & \cC_{X/Y} \ar[r] \ar[d] & \nN_{X/Y} \ar[d] \\
     & X \ar[r] & \cC_X \ar[r] & \nN_X
    }
\end{align}
with smooth vertical arrows, where $F_0 = (F^0)^\vee$ and 
\begin{align} \label{loc 4.9}
    0 \lr E_0 \lr F_0 \lr f^\ast T_Y \lr 0
\end{align} 
is an exact sequence of vector bundles. 

We now have a diagram of distinguished triangles
\begin{align*}
    \xymatrix{
    f^\ast \Omega_Y \ar[r] \ar[d] & [\fF^{-1} \to F^0] \ar[r] & [\fF^{-1} \to E^0] \ar[r] \ar[d] & f^\ast \Omega_Y[1] \ar[d] \\
    f^\ast \Omega_Y \ar[r] & \bL_{X} \ar[r] & \bL_{X/Y} \ar[r] & f^\ast \Omega_Y[1].
    }
\end{align*}

Since $[\fF^{-1} \to E^0]$ is a presentation of $\bL_{X/Y}$, the vertical arrows are isomorphisms and we obtain an induced (non-canonical) isomorphism $\bL_X \cong [\fF^{-1} \to F^0]$.

Thus, writing $\pi_1 \colon C_1 \to X$ for the projection and using the assumption that $X$ admits a closed embedding into a smooth algebraic space, \cite[Proposition~3.1]{VirRR}, \eqref{loc 4.8} and \eqref{loc 4.9} imply that
\begin{align} \label{loc 4.10}
    \tau_{C_1}(\oO_{C_1}) = \pi_1^\ast (\td (F_0)) \cap [C_1] = \pi_1^\ast (\td (E_0)) \pi_1^\ast (\td (f^\ast T_Y)) \cap [C_1]. 
\end{align}
By definition, the virtual structure sheaf of $X$ is equal to
\begin{align} \label{loc 4.11}
    [\sO_X\virt] = 0_{\Ob_X}^![\oO_{\fc_{X/Y}}] = 0_{E_1}^! [\oO_{C_1}].
\end{align}
Using the properties of the Riemann--Roch transformation, \eqref{loc 4.11} gives
\begin{align} \label{loc 4.12}
    \tau_X ( [\sO_X\virt] ) = \tau_X ( 0_{E_1}^! [\oO_{C_1}] ) = \td(-E_1) \cdot 0_{E_1}^! ( \tau_{E_1}([\oO_{C_1}]) ).
\end{align}
Denoting the embedding $C_1 \to E_1$ by $k$ and the projection $p_1 \colon E_1 \to X$, using the projection formula and $\pi_1 = p_1 \circ k$, we also have by \eqref{loc 4.10}
\begin{align} \label{loc 4.13}
    \tau_{E_1}([\oO_{C_1}])  & = k_\ast \tau_{C_1}([\oO_{C_1}]) = k_\ast ( \pi_1^\ast (\td (F_0)) \cap [C_1] ) = \\
    & = k_\ast ( k^\ast p_1^\ast (\td (F_0)) \cap [C_1] ) = \notag \\
    & =p_1^\ast (\td (F_0)) \cap k_\ast [C_1] = \notag \\
    & =p_1^\ast (\td (E_0)) p_1^\ast (\td(f^\ast T_Y)) \cap k_\ast [C_1]. \notag
\end{align}
Combining \eqref{loc 4.12} and \eqref{loc 4.13} finally yields
\begin{align*}
    \td\virt(X/Y) & = \tau_X ( [\sO_X\virt] ) =  \td(-E_1) \cdot 0_{E_1}^! \left( p_1^\ast (\td (F_0)) \cap k_\ast [C_1] \right) = \\
    & =\td(-E_1) \cdot \td(F_0) \cap [X]\virt = \frac{\td(E_0)}{\td(E_1)} \td(f^\ast T_Y) \cap [X]\virt,
\end{align*}
since $0_{E_1}^!(k_*[C_1]) = [X]\virt$ and $0_{E_1}^! \circ p_1^\ast = \id$, concluding the proof.
\end{proof}

\subsection{The equivariant case} \label{subsection 5.2} Let now $f \colon X \to Y$ be a $G$-equivariant morphism between an algebraic space $X$ with $G$-action and a smooth, pure-dimensional $G$-scheme $Y$ and $\phi$ be a $G$-equivariant almost perfect obstruction theory on $f$, inducing a closed embedding $j_\phi \colon \fc_{X/Y} \to \Ob_X$.

By Definition~\ref{equivariant vir str sheaf and cycle def}, $[\sO_X\virt]$ is then naturally an element of $K_0^G(X)$.

\begin{defi}
The \emph{equivariant virtual Todd class} of $X$ over $Y$ is the class
\begin{align}
    \td^{\mathrm{vir}, G}(X/Y) := \tau_X^G( [\sO_X\virt] ) \in A_*^G(X),
\end{align}
where $\tau_X^G$ is the equivariant Riemann--Roch transformation of $X$.
\end{defi}

As before, we deduce a tautological equivariant virtual Riemann--Roch formula from the definition.

\begin{thm} \label{thm:equiv GRR formula}
When $f$ is proper, we have for any $V \in K_G^0(X)$
\begin{align*}
    \mathrm{ch}^G(f_\ast(V\otimes [\sO_X\virt])) \cdot \td^G(T_Y) \cap [Y] = f_\ast \left( \mathrm{ch}^G (V) \cdot \td^{\mathrm{vir},G}(X/Y) \right).
\end{align*}
\end{thm}

We would now like to generalize Theorem~\ref{virtual todd thm} to the present equivariant context. We will achieve this by reducing to the non-equivariant case as follows. Let $(V,U)$ be an $l$-dimensional good pair for the integer $n-i$. Recall that $\tau_X^G$ is defined componentwise by the diagram~\eqref{fund diagram equivariant rr}
\begin{align*}
    \xymatrix{
    K_0^G(X \times V) \ar[r]^-{\jmath^\ast} & K_0^G(X \times U) \ar[r] & K_0(X \times^G U) \ar[d]^-{\left( \frac{\tau_{X \times^G U}}{\td(X \times^G (U \times V))} \right)_{i+l-g}} \\
    K_0^G(X) \ar[u]^-{\pi_X^\ast} \ar[ur]^-{\pi_X^\ast} \ar[rr]_-{(\tau_X^G)_i} \ar[urr]_-{s_U} & & A_{i+l-g}(X \times^G U).
    }
\end{align*}

Write $\pi_V$ for both the vector bundle projections $X\times V \to V$ and $X\times^G (U \times V) \to X\times^G U$. By abuse of notation, we use the letter $V$ to denote the latter vector bundle on $X \times^G U$. 

We also write $s_U (\fF)$ for the sheaf (or complex of sheaves) on $X \times^G U$ induced by pulling back a $G$-equivariant sheaf $\fF$ (or complex of sheaves) on $X$ via $\pi_X$ and descending to $X \times^G U$. Similarly, we can define $s_U (B) \in A_*(X \times^G U)$ for any class $B \in A_*^G(X)$.

We may work with $X \times^G U$ by the following proposition.

\begin{prop} \label{prop 4.11}
A $G$-equivariant almost perfect obstruction theory $\phi$ on $f \colon X \to Y$ induces a canonically defined almost perfect obstruction theory $\phi_U$ on the morphism $f_U \colon X \times^G U \to Y \times^G U$. 

Their virtual structure sheaves and virtual fundamental cycles satisfy
\begin{align} \label{loc 4.15}
    s_U ( [\sO_X\virt] ) & = [\sO_{X \times^G U}\virt] \in K_0(X \times^G U), \\
    s_U ( [X]\virt ) & = [X \times^G U]\virt \in A_*(X \times^G U). \notag
\end{align}
\end{prop}

\begin{proof}
For any index $\alpha$, pulling back the $G$-equivariant perfect obstruction theory $\phi\lalp \colon E\lalp \to \bL_{X\lalp/Y}$ via the projection $\pi_X \colon X \times U \to X$ gives a morphism 
\begin{align} \label{loc 5.16}
     \pi_X^\ast E\lalp \lr  \pi_X^\ast \bL_{X\lalp/Y}. 
\end{align}
Since $\bL_{X \times U / Y \times U} = \pi_X^\ast \bL_{X/Y}$ descends to $\bL_{X \times^G U / Y \times^G U}$ on $X \times^G U$, \eqref{loc 5.16} descends to give a perfect obstruction theory
\begin{align} \label{loc 4.17}
    (\phi_U)\lalp \colon (E_U)\lalp := s_U(E\lalp) \lr \bL_{X\lalp \times^G U / Y \times^G U}.
\end{align}

It is a formal exercise to verify that these morphisms form part of the data of an almost perfect obstruction theory $\phi_U$ on $X \times^G U$ by pulling back the isomorphisms $\psi\lab$ and $\eta_{\alpha\beta\lambda}$  in Definition~\ref{equivariant APOT} via $\pi_X$ and then descending to $X \times^G U$. The obstruction sheaf of $\phi_U$ is $s_U(\Ob_X)$.

We have that $\fc_{X/Y} \times_X (X \times U) = \fc_{X \times U / Y \times U}$. The equality \eqref{loc 4.15} then follows immediately from the definition of the Gysin map $0_{\Ob_X}^{!,G}$ in Theorem~\ref{equivariant gysin map k-theory alg spc} using an affine $G$-equivariant \'{e}tale atlas $V \to X \times U$, the definition of $0_{s_U(\Ob_X)}^!$ and \eqref{loc 4.17}.
\end{proof}

A $G$-equivariant global presentation of the virtual tangent bundle of $f$ can be defined in analogy with Definition~\ref{def of presentation of virtual tangent}.

\begin{defi} \label{equiv global prep of vir tangent}
A \emph{$G$-equivariant global presentation of the virtual tangent bundle} of $f$ is a global presentation in the sense of Definition~\ref{def of presentation of virtual tangent} such that the complex $E$ and the morphism $E^0 \to \bL_{X/Y}$ are $G$-equivariant.
\end{defi}

\begin{prop} \label{prop 4.12}
A $G$-equivariant global presentation $E^\vee = [E_0 \to E_1]$ of the virtual tangent bundle of $f$ induces a global presentation $F^\vee = [F_0 \to F_1]$ of the virtual tangent bundle of $f_U$, where $F_i = s_U (E_i)$.
\end{prop}

\begin{proof}
This is immediate from Proposition~\ref{prop 4.11} and the definitions.
\end{proof}

\begin{thm} \label{vir todd class formula thm}
Suppose that $X$ admits a $G$-invariant closed embedding into a smooth algebraic space with a $G$-action and let $E^\vee = [E_0 \to E_1]$ be a $G$-equivariant global presentation of the virtual tangent bundle of $f$. Then
\begin{align*}
    \td^{\mathrm{vir},G}(X/Y) = \frac{\td^G(E_0)}{\td^G(E_1)} \cdot \frac{\td^G(f^\ast T_Y)}{\td^G(\fG)} \cap [X]\virt = \td^G(T_{X/Y}\virt) \cap [X]\virt,
\end{align*}
where $T_{X/Y}\virt = [E_0] - [E_1] + [f^\ast T_Y] - [\fG] \in K_G^0(X)$ and $\fG$ is the adjoint representation of $G$.
\end{thm}

\begin{proof}
A $G$-invariant closed embedding of $X$ into a smooth algebraic space with a $G$-action induces a closed embedding of any $X \times^G U$ into a smooth algebraic space for any good pair $(V,U)$. By Propositions~\ref{prop 4.11}, \ref{prop 4.12} and Theorem~\ref{virtual todd thm}, we have for the almost perfect obstruction theory $\phi_U$
\begin{align} \label{loc 4.18}
    \tau_{X \times^G U}([\sO_{X \times^G U}\virt])  = \frac{\td(F_0)}{\td(F_1)} \td(f_U^\ast T_{Y \times^G U}) \cap [X \times^G U]\virt,
\end{align}
and hence, using \eqref{loc 4.15},
\begin{align} \label{loc 4.19}
    (\tau_X^G)_i([\sO_X\virt]) & = \left( \frac{\tau_{X \times^G U}}{\td(X \times^G (U \times V))} \right)_{i+l-g}(s_U([\sO_X\virt])) \\
    & = \left( \frac{\tau_{X \times^G U}}{\td(X \times^G (U \times V))} \right)_{i+l-g}([\sO_{X \times^G U}\virt]). \notag
\end{align}

Since $F_i = s_U(E_i)$, for any $B \in A_*^G(X)$ we have 
\begin{align} \label{loc 4.20}
    \td(F_i) \cap s_U(B) = \td^G(E_i) \cap B.
\end{align}

Moreover, $T_{Y \times^G U}$ is the descent of the equivariant perfect complex of amplitude $[-1,0]$
$$ \fG \lr \pi_Y^\ast T_Y \oplus \pi_U^\ast T_U $$
on $Y \times U$. Thus, since $T_U$ is equivariantly isomorphic to the vector bundle $V$ over $U$, we obtain
\begin{align} \label{loc 4.21}
    \td(f_U^\ast T_{Y \times^G U}) = \frac{\td^G(f^\ast T_Y)}{\td^G(f^\ast \fG)} \cdot \td^G(V).
\end{align}

Combining \eqref{loc 4.18}, \eqref{loc 4.19}, \eqref{loc 4.20} and \eqref{loc 4.21}, we get
\begin{align*}
    (\tau_X^G)_i([\sO_X\virt]) = \left( \frac{\td^G(E_0)}{\td^G(E_1)} \cdot \frac{\td^G(f^\ast T_Y)}{\td^G(\fG)} \cap [X]\virt \right)_i
\end{align*}
for any $i$, which concludes the proof.
\end{proof}

\subsection{An application in generalized Donaldson--Thomas theory} We conclude this section with an application of our results.
\medskip

Let $\mM$ be the moduli stack parametrizing Gieseker semistable coherent sheaves of a fixed Chern character on a smooth, projective Calabi--Yau threefold (cf. \cite{HuyLehn} for more background on Gieseker stability and moduli of sheaves). This is a global quotient stack of the form $\mM = [Q / G]$ arising from Geometric Invariant Theory (GIT) \cite{MFK}. As such, it admits a closed embedding into a smooth quotient stack $\aA = [P / G]$.
\medskip

In \cite{KLS}, the authors construct a canonical, proper \DM stack $\tilde{\mM} = [\tilde{Q} / G]$, called the intrinsic stabilizer reduction of $\mM$. By definition, it is equipped with a natural projection map $\pi \colon \tilde{\mM} \to \mM$ and a closed embedding into the smooth, proper \DM stack $\tilde{\aA} = [\tilde{P} / G]$, the intrinsic stabilizer reduction of $\aA$, which is also a GIT global quotient stack.

In addition, it is shown in \cite{KLS, KiemSavvas} that $\tilde{\mM}$ admits a natural almost perfect obstruction theory $\phi$, even though it is not quasi-smooth and thus does not carry a natural perfect obstruction theory. This consists of a $G$-equivariant \'{e}tale covering $\lbrace \tilde{Q}_\alpha \to \tilde{Q} \rbrace$, inducing an \'{e}tale covering $\lbrace \tilde{\mM}\lalp := [\tilde{Q}\lalp / G] \to \tilde{\mM} \rbrace$, together with perfect obstruction theories $\phi\lalp \colon E\lalp \to \bL_{\tilde{\mM}\lalp}$ of virtual dimension zero, induced by $G$-equivariant perfect obstruction theories $\psi\lalp \colon F\lalp \to \bL_{\tilde{Q}\lalp}$ by means of commutative diagrams
\begin{align*}
\xymatrix{
q\lalp^\ast E\lalp \ar[r] \ar[d]_-{q\lalp^\ast \phi\lalp} & F\lalp \ar[d]_-{\psi\lalp} \ar[r] & \fG^\vee \ar@{=}[d] \\
q\lalp^\ast \bL_{\tilde{\mM}\lalp} \ar[r] & \bL_{\tilde{Q}\lalp} \ar[r] & \fG^\vee,
}
\end{align*}
where the rows are exact triangles and $q\lalp \colon \tilde{Q}\lalp \to \tilde{\mM}\lalp$ is the quotient map.

These in turn are part of the data of a $G$-equivariant almost perfect obstruction theory $\psi$ on $\tilde{Q}$ and hence an almost perfect obstruction theory on the natural morphism $\tilde{\mM} \to BG$.

It is routine to check that the obstruction sheaves, virtual structure sheaves and virtual fundamental cycles induced on $\tilde{\mM}$ by $\phi$ and $\psi$ are equal and we denote them by $\Ob_{\tilde{\mM}}$, $[\sO_{\tilde{\mM}}\virt]$ and $[\tilde{\mM}]\virt$ respectively, as usual. Their $G$-equivariant counterparts on $\tilde{Q}$ are denoted in an analogous fashion. 
\medskip

Now, ${\mM}$ is also the classical truncation of a $(-1)$-shifted symplectic derived Artin stack $\boldsymbol{\mM}$ (cf. \cite{PTVV}). Hence, the results of \cite{HekSav} imply the existence of a canonical derived enhancement $\boldsymbol{\tilde{\mM}}$ of $\tilde{\mM}$ such that the $[0,1]$-truncation of $\bT_{\boldsymbol{\tilde{\mM}}}|_{\tilde{\mM}}$, the restriction of the derived tangent complex of $\boldsymbol{\tilde{\mM}}$ to its classical truncation $\tilde{\mM}$, is perfect. Denote this perfect complex by $E^\vee$ and its dual by $E$. 

$E$ is compatible with the almost perfect obstruction theory $\phi$ of $\tilde{\mM}$ in the sense that $E\lalp = E|_{\tilde{\mM}\lalp}$ and $h^1(E^\vee) = \Ob_{\tilde{\mM}}$ gives gluing data for the local obstruction sheaves $h^1(E\lalp^\vee)$. By definition, $h^0(E)=\Omega_{\tilde{\mM}}$.
\medskip

By construction, using the closed embedding $\tilde{\mM} \to \tilde{\aA}$, we may write $E^\vee = [E_0 \to E_1]$, where $E_0 = T_{\tilde{\aA}}|_{\tilde{\mM}}$ and $E_1$ is a locally free sheaf on $\tilde{\mM}$. Similarly, $\bL_{\tilde{\mM}}$ takes the form $[\fF^{-1} \to \Omega_{\tilde{A}}|_{\tilde{\mM}}]$, where $\fF^{-1}$ is the conormal sheaf of the closed embedding $\tilde{\mM} \to \tilde{\aA}$. In particular, there is a natural morphism $E^0 \to \bL_{\tilde{\mM}}$ such that $E^0 \to h^0(E) = \Omega_{\tilde{\mM}}$ factors through the map $\bL_{\tilde{\mM}} \to h^0(\bL_{\tilde{\mM}}) = \Omega_{\tilde{\mM}}$.
\medskip

In a similar vein, we may define $F^\vee = [F_0 \to F_1]$ with dual $F$, where $F_0 = T_{\tilde{P}}|_{\tilde{Q}}$ and $F_1$ is the $G$-equivariant bundle on $\tilde{Q}$ corresponding to the locally free sheaf $E_1$ on $\tilde{\mM}$. Like $E^\vee$, $F^\vee$ is compatible with the almost perfect obstruction theory $\psi$, satisfies $h^1(F^\vee) = \Ob_{\tilde{Q}}$, $h^0(F) = \Omega_{\tilde{Q}}$, and there is a natural $G$-equivariant morphism $F^0 \to \bL_{\tilde{Q}}$ such that $F^0 \to h^0(F) = \Omega_{\tilde{Q}}$ factors through $\bL_{\tilde{Q}} \to h^0(\bL_{\tilde{Q}})$.
\medskip

We have thus established the following.

\begin{prop}
$E^\vee$ is a global presentation of the virtual tangent bundle of $\tilde{\mM}$ with respect to the almost perfect obstruction theory $\phi$ and $F^\vee$ is a global presentation of the virtual tangent bundle of $\tilde{Q}$ with respect to the $G$-equivariant almost perfect obstruction theory $\psi$.
\end{prop}

Writing $F_0'$ for the $G$-equivariant bundle on $\tilde{Q}$ corresponding to the locally free sheaf $E_0$ on $\tilde{\mM}$, we have a short exact sequence
$$0 \lr \fG \lr F_0 \lr F_0' \lr 0$$
and therefore $\td^G(F_0) = \td^G(F_0') \td^G(\fG)$. Thus Theorem~\ref{vir todd class formula thm}, applied to the $G$-equivariant morphism $\tilde{Q} \to \mathrm{pt}$, implies that 
\begin{align*}
\td^{\mathrm{vir}, G} ( \tilde{Q} ) & = \td^{G}(T_{\tilde{Q}}\virt) \cap [\tilde{Q}]\virt = \\
& = \frac{\td^G(F_0)}{\td^G(F_1)} \cdot \frac{\td^G(f^\ast T_{\mathrm{pt}})}{\td^G(\fG)} \cap [\tilde{Q}]\virt = \frac{\td^G(F_0')}{\td^G(F_1)} \cap [\tilde{Q}]\virt 
\end{align*}
and hence, passing back to $\tilde{\mM}$, we may write
$$\td\virt(\tilde{\mM}) = \frac{\td(E_0)}{\td(E_1)} \cap [\tilde{\mM}]\virt = \td(T_{\tilde{\mM}}\virt) \cap [\tilde{\mM}]\virt.$$

Therefore, applying Theorem~\ref{thm:equiv GRR formula} to the $G$-equivariant morphism $\tilde{Q} \to \mathrm{pt}$, we obtain the following result.

\begin{thm}
For any $V \in K_G^0(\tilde{Q}) \cong K^0(\tilde{\mM})$, we have
\begin{align*}
\chi^G ( [\oO_{\tilde{Q}}\virt] \otimes V) = \int_{[\tilde{Q}]\virt} \ch^G(V) \td^{G}(T_{\tilde{Q}}\virt)
\end{align*}
or, equivalently,
\begin{align} \label{loc 5.22}
\chi ( [\oO_{\tilde{\mM}}\virt] \otimes V) = \int_{[\tilde{\mM}]\virt} \ch(V) \td(T_{\tilde{\mM}}\virt).
\end{align}
\end{thm}

In particular, since the virtual dimension of $\tilde{\mM}$ is zero, we obtain the following corollary by taking $V = \sO_{\tilde{\mM}}$.

\begin{cor}
$\chi ( [\oO_{\tilde{\mM}}\virt]) = \int_{[\tilde{\mM}]\virt} 1$.
\end{cor}

The left-hand side was defined in \cite{KiemSavvas} to be the (un-twisted) $K$-theoretic generalized Donaldson--Thomas invariant via Kirwan blowups associated to $\mM$, whereas the right-hand side the intersection-theoretic generalized Donaldson--Thomas invariant via Kirwan blowups associated to $\mM$ in \cite{KLS}. We have thus shown that these two invariants are equal, as expected.

More generally, the left-hand side of~\eqref{loc 5.22} is the $V$-twisted $K$-theoretic generalized Donaldson--Thomas invariant via Kirwan blowups associated to $\mM$. Another immediate consequence of~\eqref{loc 5.22} is its deformation invariance.

\section{Cosection Localization} \label{cos loc section}

In this last section, we examine the compatibility of the Riemann--Roch transformation, at the levels of generality considered earlier in the paper, with localization by cosection \cite{KiemLiCosection, KiemLiKTheory, KiemSavvasLoc}, a standard tool in the treatment of virtual cycles and virtual structure sheaves. Namely, we prove cosection localized versions of the Riemann--Roch theorems established in the preceding section.
\medskip

To fix notation, let $f \colon X \to Y$ be a $G$-equivariant morphism from an algebraic space $X$ with $G$-action to a smooth, pure-dimensional $G$-scheme $Y$ and suppose that $f$ is equipped with a $G$-equivariant almost perfect obstruction theory $\phi$. We also assume given a $G$-invariant cosection $\sigma \colon \Ob_X \to \oO_X$ and write $X(\sigma)$ for its vanishing locus.
\medskip

As usual, we begin with the non-equivariant case where the group $G$ is trivial. 

By \cite{KiemLiCosection, KiemLiKTheory, KiemSavvasLoc}, the existence of the cosection $\sigma$ implies that the virtual cycle $[X]\virt \in A_\ast (X)$ and virtual structure sheaf $[\oO_X\virt] \in K_0(X)$ can be naturally localized to the locus $X(\sigma)$. We thus obtain their cosection localized counterparts $[X]_{\loc}\virt \in A_\ast (X(\sigma))$ and $[\oO_{X, \loc}\virt] \in K_0(X(\sigma))$.

 The following theorem is a cosection localized Riemann--Roch theorem.

\begin{thm} \label{cosection localized non-equivariant vir RR}
Suppose that $X$ admits a closed embedding into a smooth algebraic space and let $E^\vee = [E_0 \to E_1]$ be a global presentation of the virtual tangent bundle of $f$. Then
\begin{align*}
    \tau_{X(\sigma)} ([\oO_{X, \loc}\virt]) = \td(E^\vee) \td(f^* T_Y) \cap [X]_\loc\virt.
\end{align*}
\end{thm}

\begin{proof}
The proof follows verbatim the notation and steps of the proof of Theorem~\ref{virtual todd thm}. We explain the necessary modifications.

\eqref{loc 4.10} still holds and thus we have
\begin{align*}
    \tau_{C_1} (\oO_{C_1}) = \pi_1^\ast (\td(E_0)) \pi_1^\ast (\td(f^\ast T_Y)) \cap [ C_1 ].
\end{align*}

Let $E_1(\sigma) = E_1|_{X(\sigma)} \cup \ker \left( \sigma|_{U(\sigma)} \right)$ considered as a closed algebraic subspace of the vector bundle $E_1$. The reduced support of $C_1$ lies in $E_1(\sigma)$ and, by definition, the cosection localized virtual cycle and virtual structure sheaf of $X$ are defined to be
\begin{align*}
    [X]_\loc\virt & = 0_{\Ob_X, \sigma}^! [\fc_{X/Y}] = 0_{E_1, \sigma}^! [C_1] \in A_\ast (X(\sigma)), \\
    [\oO_{X,\loc}\virt] & = 0_{\Ob_X, \sigma}^! [\sO_{\fc_{X/Y}}] = 0_{E_1, \sigma}^! [\sO_{C_1}] \in K_0(X(\sigma)),
\end{align*}
using the cosection localized Gysin maps in intersection theory and $K$-theory respectively, where $[\sO_{C_1}]$ can naturally be considered as an element of $K_0(E_1(\sigma))$.

To complete the proof, we need the cosection localized analogue of \eqref{loc 4.12}:
\begin{align*}
    \tau_{X(\sigma)}([\oO_{X, \loc}\virt]) = \td(-E_1) \cdot 0_{E_1, \sigma}^!(\tau_{E_1(\sigma)}([\sO_{C_1}])).
\end{align*}
But this is the special case $[F]=[\sO_{C_1}]$ of formula~(5.21) in the proof of \cite[Theorem~5.8]{KiemLiKTheory}.
\end{proof}

An immediate corollary is the following cosection localized Riemann--Roch formula.

\begin{thm}
When the restriction $f' = f|_{X(\sigma)} \colon X(\sigma) \to Y$ is proper, we have for any $V \in K^0(X(\sigma))$
\begin{align*}
     \mathrm{ch}(f_\ast'(V\otimes [\sO_{X,\loc}\virt])) \cdot \td(T_Y) \cap [Y] = f_\ast' \left( \mathrm{ch}(V) \cdot \td\virt(X/Y) \cap [X]_\loc\virt \right).
\end{align*}
\end{thm}

Suppose now that the group $G$ is general. Our treatment of the equivariant Riemann--Roch theorem in Subsection~\ref{subsection 5.2} reduces the statement to the non-equivariant case. The following proposition allows us to follow the steps of Subsection~\ref{subsection 5.2} verbatim in the presence of a cosection.

\begin{prop}
Let $(V,U)$ be an $l$-dimensional good pair for the integer $n-i$ and $\phi_U$ be the induced almost perfect obstruction theory for the morphism $f_U \colon X \times^G U \to Y\times^G U$ (cf. Proposition~\ref{prop 4.11}).
Then, a $G$-invariant cosection $\sigma$ for $\phi$ induces a canonically defined cosection $\sigma_U = s_U (\sigma)$ for $\phi_U$ with vanishing locus $(X \times^G U) (\sigma_U) = X(\sigma) \times^G U$.
\end{prop}

\begin{proof}
The statement follows immediately, since the obstruction sheaf of $\phi_U$ is $s_U(\Ob_X)$ and we also have $s_U(\sO_X) = \sO_{X \times^G U}$.
\end{proof}

Using Theorem~\ref{cosection localized non-equivariant vir RR}, the proof of Theorem~\ref{vir todd class formula thm} goes through following the steps line by line to yield the following cosection localized virtual Riemann--Roch theorem.

\begin{thm}
Suppose that $X$ admits a $G$-invariant closed embedding into a smooth algebraic space with a $G$-action and let $E^\vee = [E_0 \to E_1]$ be a $G$-equivariant global presentation of the virtual tangent bundle of $f$. Then
\begin{align*}
    \tau_{X(\sigma)}^G ([\sO_{X,\loc}\virt]) = \frac{\td^G(E_0)}{\td^G(E_1)} \cdot \frac{\td^G(f^\ast T_Y)}{\td^G(\fG)} \cap [X]_{\loc}\virt \in A_\ast^G(X(\sigma)),
\end{align*}
where we set $T_{X/Y}\virt = [E_0] - [E_1] + [f^\ast T_Y] - [\fG] \in K_G^0(X)$ and $\fG$ is the adjoint representation of $G$.
\end{thm}

\bibliography{Master}

\begin{thebibliography}{10}

\bibitem{Alper}
Jarod Alper, Jack Hall, and David Rydh.
\newblock A {L}una \'{e}tale slice theorem for algebraic stacks.
\newblock {\em Ann. of Math. (2)}, 191(3):675--738, 2020.

\bibitem{BehGW}
Kai Behrend.
\newblock Gromov-{W}itten invariants in algebraic geometry.
\newblock {\em Invent. Math.}, 127(3):601--617, 1997.

\bibitem{BehFan}
Kai Behrend and Barbara Fantechi.
\newblock The intrinsic normal cone.
\newblock {\em Invent. Math.}, 128(1):45--88, 1997.

\bibitem{LiChang}
Huai-{L}iang Chang and Jun Li.
\newblock Semi-perfect obstruction theory and {D}onaldson-{T}homas invariants
  of derived objects.
\newblock {\em Comm. Anal. Geom.}, 19(4):807--830, 2011.

\bibitem{EdidinGrahamEquivIT}
Dan Edidin and William Graham.
\newblock Equivariant intersection theory.
\newblock {\em Invent. Math.}, 131(3):595--634, 1998.

\bibitem{EdidinGrahamRR}
Dan Edidin and William Graham.
\newblock Riemann-{R}och for equivariant {C}how groups.
\newblock {\em Duke Math. J.}, 102(3):567--594, 2000.

\bibitem{VirRR}
Barbara Fantechi and Lothar G\"{o}ttsche.
\newblock Riemann-{R}och theorems and elliptic genus for virtually smooth
  schemes.
\newblock {\em Geom. Topol.}, 14(1):83--115, 2010.

\bibitem{Fulton}
William Fulton.
\newblock {\em Intersection theory}, volume~2 of {\em Ergebnisse der Mathematik
  und ihrer Grenzgebiete. 3. Folge. A Series of Modern Surveys in Mathematics
  [Results in Mathematics and Related Areas. 3rd Series. A Series of Modern
  Surveys in Mathematics]}.
\newblock Springer-Verlag, Berlin, second edition, 1998.

\bibitem{Gillet}
Henri Gillet.
\newblock Intersection theory on algebraic stacks and {$Q$}-varieties.
\newblock In {\em Proceedings of the {L}uminy conference on algebraic
  {$K$}-theory ({L}uminy, 1983)}, volume~34, pages 193--240, 1984.

\bibitem{Gross}
Philipp Gross.
\newblock Tensor generators on schemes and stacks.
\newblock {\em Algebr. Geom.}, 4(4):501--522, 2017.

\bibitem{HekSav}
Jeroen Hekking, David Rydh, and Michail Savvas.
\newblock Stabilizer reduction for derived stacks and applications to
  sheaf-theoretic invariants.
\newblock {\em arXiv e-prints}, page arXiv:2209.15039, September 2022.

\bibitem{HuyLehn}
Daniel Huybrechts and Manfred Lehn.
\newblock {\em The geometry of moduli spaces of sheaves}.
\newblock Cambridge Mathematical Library. Cambridge University Press,
  Cambridge, second edition, 2010.

\bibitem{Inaba}
Michi-aki Inaba.
\newblock Toward a definition of moduli of complexes of coherent sheaves on a
  projective scheme.
\newblock {\em J. Math. Kyoto Univ.}, 42(2):317--329, 2002.

\bibitem{Joshua1}
Roy Joshua.
\newblock Bredon-style homology, cohomology and {R}iemann-{R}och for algebraic
  stacks.
\newblock {\em Adv. Math.}, 209(1):1--68, 2007.

\bibitem{KhanRR}
Adeel~A. {Khan}.
\newblock {Virtual fundamental classes of derived stacks I}.
\newblock {\em arXiv e-prints}, page arXiv:1909.01332, September 2019.

\bibitem{KiemLiCosection}
Young-Hoon Kiem and Jun Li.
\newblock Localizing virtual cycles by cosections.
\newblock {\em J. Amer. Math. Soc.}, 26(4):1025--1050, 2013.

\bibitem{KiemLiKTheory}
Young-Hoon Kiem and Jun Li.
\newblock Localizing virtual structure sheaves by cosections.
\newblock {\em Int. Math. Res. Not. IMRN}, (22):8387--8417, 2020.

\bibitem{KLS}
Young-Hoon {Kiem}, Jun {Li}, and Michail {Savvas}.
\newblock {Generalized Donaldson-Thomas Invariants via Kirwan Blowups}.
\newblock {\em arXiv e-prints}, page arXiv:1712.02544, December 2017.

\bibitem{KiemSavvasLoc}
Young-Hoon Kiem and Michail Savvas.
\newblock Localizing virtual structure sheaves for almost perfect obstruction
  theories.
\newblock {\em Forum of Mathematics, Sigma}, 8:e61, 2020.

\bibitem{KiemSavvas}
Young-Hoon Kiem and Michail Savvas.
\newblock {$K$}-theoretic generalized {D}onaldson-{T}homas invariants.
\newblock {\em Int. Math. Res. Not. IMRN}, (3):2123--2158, 2022.

\bibitem{KreschCycles}
Andrew Kresch.
\newblock Cycle groups for {A}rtin stacks.
\newblock {\em Invent. Math.}, 138(3):495--536, 1999.

\bibitem{yplee}
Yuan-Pin Lee.
\newblock Quantum {$K$}-theory. {I}. {F}oundations.
\newblock {\em Duke Math. J.}, 121(3):389--424, 2004.

\bibitem{LiTian}
Jun Li and Gang Tian.
\newblock Virtual moduli cycles and {G}romov-{W}itten invariants of algebraic
  varieties.
\newblock {\em J. Amer. Math. Soc.}, 11(1):119--174, 1998.

\bibitem{MFK}
David Mumford, James Fogarty, and Frances Kirwan.
\newblock {\em Geometric invariant theory}, volume~34 of {\em Ergebnisse der
  Mathematik und ihrer Grenzgebiete (2) [Results in Mathematics and Related
  Areas (2)]}.
\newblock Springer-Verlag, Berlin, third edition, 1994.

\bibitem{Okou2}
Andrei Okounkov.
\newblock Lectures on {K}-theoretic computations in enumerative geometry.
\newblock In {\em Geometry of moduli spaces and representation theory},
  volume~24 of {\em IAS/Park City Math. Ser.}, pages 251--380. Amer. Math.
  Soc., Providence, RI, 2017.

\bibitem{Okou1}
Andrei Okounkov.
\newblock Takagi lectures on {D}onaldson-{T}homas theory.
\newblock {\em Jpn. J. Math.}, 14(1):67--133, 2019.

\bibitem{PT1}
Rahul Pandharipande and Richard~P. Thomas.
\newblock Curve counting via stable pairs in the derived category.
\newblock {\em Invent. Math.}, 178(2):407--447, 2009.

\bibitem{PTVV}
Tony Pantev, Bertrand To{\"e}n, Michel Vaqui{\'e}, and Gabriele Vezzosi.
\newblock Shifted symplectic structures.
\newblock {\em Publications math{\'e}matiques de l'IH{\'E}S}, 117(1):271--328,
  2013.

\bibitem{EquivVirRR}
Charanya Ravi and Bhamidi Sreedhar.
\newblock Virtual equivariant {G}rothendieck-{R}iemann-{R}och formula.
\newblock {\em Doc. Math.}, 26:2061--2094, 2021.

\bibitem{Romagny}
Matthieu Romagny.
\newblock Group actions on stacks and applications.
\newblock {\em Michigan Math. J.}, 53(1):209--236, 2005.

\bibitem{Sav}
Michail {Savvas}.
\newblock {Generalized Donaldson-Thomas Invariants of Derived Objects via
  Kirwan Blowups}.
\newblock {\em arXiv e-prints}, page arXiv:2005.13768, May 2020.

\bibitem{Thomas}
Richard~P. Thomas.
\newblock A holomorphic {C}asson invariant for {C}alabi-{Y}au 3-folds, and
  bundles on {$K3$} fibrations.
\newblock {\em J. Differential Geom.}, 54(2):367--438, 2000.

\bibitem{ThomasonRes}
R.~W. Thomason.
\newblock Equivariant resolution, linearization, and {H}ilbert's fourteenth
  problem over arbitrary base schemes.
\newblock {\em Adv. in Math.}, 65(1):16--34, 1987.

\bibitem{ToenDMRR}
B.~Toen.
\newblock Th\'{e}or\`emes de {R}iemann-{R}och pour les champs de
  {D}eligne-{M}umford.
\newblock {\em $K$-Theory}, 18(1):33--76, 1999.

\bibitem{Totaro}
Burt Totaro.
\newblock The resolution property for schemes and stacks.
\newblock {\em J. Reine Angew. Math.}, 577:1--22, 2004.

\end{thebibliography}
\bibliographystyle{plain}

\end{document}